\documentclass[12pt,a4paper]{article}
\usepackage{amssymb}
\usepackage{amsmath,amsfonts,amsthm}

\usepackage{graphicx}
\usepackage{amsmath}

\usepackage{amsfonts}
\usepackage{amsthm,amscd}
\usepackage{graphicx}
\usepackage{amsmath}
\usepackage{amssymb}
\usepackage{amstext}

\newtheorem{theorem}{Theorem}
\newtheorem{corollary}[theorem]{Corollary}
\newtheorem{definition}[theorem]{Definition}
\newtheorem{example}[theorem]{Example}
\newtheorem{lemma}[theorem]{Lemma}
\newtheorem{notation}[theorem]{Notation}
\newtheorem{proposition}[theorem]{Proposition}
\newtheorem{remark}[theorem]{Remark}

\title{On information gain, Kullback-Leibler divergence, entropy production and the involution kernel}

\author{A. O. Lopes  and J. K. Mengue}

\begin{document}

\maketitle

\begin{abstract}
It is well known that in Information Theory and  Machine Learning  the Kullback-Leibler divergence, which extends the concept of Shannon entropy, plays a fundamental role. Given an {\it a priori} probability kernel $\hat{\nu}$ and a probability $\pi$ on the measurable space $X\times Y$ we  consider an appropriate definition of entropy of $\pi$ relative to $\hat{\nu}$, which is based on previous works. Using this concept of entropy we obtain a natural definition of information gain for general measurable spaces which coincides with the mutual information given from the K-L divergence in the case $\hat{\nu}$ is identified with a probability $\nu$ on $X$. This will be used to extend the meaning of specific information gain and dynamical entropy production to the model of thermodynamic formalism for symbolic dynamics over a compact alphabet (TFCA model). In this case, we show that the involution kernel is a natural tool for better understanding some important properties of entropy production.



\end{abstract}

\section{Introduction}

The main goal of this paper is to introduce and study the concepts of information gain and entropy production to equilibrium measures in symbolic dynamics over a compact space (rather than finite) alphabet. In this way, the first part of this work lies in the frontier between Information Theory and Ergodic Theory. In the second part of the paper, we consider such model of thermodynamic formalism for symbolic dynamics over a compact alphabet (which we will abbreviate by TFCA model) in Ergodic Theory (see \cite{LMMS}).

We start by introducing some elements of Information Theory.  In Data Compression  the Shannon entropy\footnote{we will consider here $\log(x) = \ln(x)$,  but any basis could be also used. Furthermore, $0\log(0) = 0$, by convention. }, $S(P) = - \sum_{i=1}^{d}p_i\log(p_i),$ of a probability vector $P=(p_1,...,p_d)$ plays an important role (see \cite{S} and \cite{CT} chap. 5).  For the benefit of the reader we exhibit introductory examples concerning $S(P)$ in the appendix section \ref{appendix1}. 

Related to this, in the study of Decision Trees in Machine Learning it is also considered another important concept, called \textbf{information gain}. Following \cite{Q} (see p. 89-90), for a probability $\pi$ on $X\times Y=\{1,...,d\}\times\{1,...,r\}$ with $x-$marginal $P=(p_1,...,p_d)$, we define the information gain of $\pi$ with respect to $P$ as
\begin{equation}\label{IG}
	IG(\pi,P) =  \underbrace{-\sum_{x=1}^{d}p_x\log(p_x)}_{S(P)} - \underbrace{\sum_{y=1}^{r}q_y\left[-\sum_{x=1}^{d} \frac{\pi_{x,y}}{q_y}\log\left(\frac{\pi_{x,y}}{q_y}\right)\right]}_{H(\pi)},
\end{equation}
where $q_y = \sum_x \pi_{x,y}$, that is, $Q=(q_1,...,q_r)$ is the $y$-marginal of $\pi$.
In this expression the number $$-\left[\sum_{x=1}^{d} \frac{\pi_{x,y}}{q_y}\log(\frac{\pi_{x,y}}{q_y})\right]$$ is the Shannon entropy of the probability obtained from the distribution of $\pi$ on the line $X\times\{y\}$ and, therefore,  $H(\pi)$ is just the weighted mean of these entropies according to $Q$.  Example \ref{exampleIG} (in our appendix
section    \ref {appendix1}) will exhibit   a concrete interpretation of $IG(\pi,P)$.

Denoting by  $P$ and $Q$ the marginals of $\pi$, we get that the number $IG(\pi,P)$ can be rewritten as
\[ \sum_{x=1}^{d}\sum_{y=1}^{r} {\pi_{x,y}}\log\left(\frac{\pi_{x,y}}{p_xq_y}\right), \]
which, in Information Theory, is called of \textbf{mutual information} 
  (see \cite{CT}).

In Ergodic Theory, for the case of the symbolic space
\begin{equation} \label{jai}\Omega=\Omega^+=\{1,2,...,d\}^{\mathbb{N}}=\{|x_1,x_2,x_3,...)\,|\, x_i \in \{1,2,...,d\} \,,\forall i \in \mathbb{N}\}, \end{equation}
it is considered the Kolmogorov-Sinai entropy for stationary probabilities, that means, probabilities $\mu$ on $\Omega$ which are invariant by the shift map $\sigma:\Omega\to \Omega$,  $\sigma(|x_1,x_2,x_3,...)) = |x_2,x_3,x_4,...)$.
The set $\Omega$ is a compact metric space, when equipped with the metric
$d(|x_1,x_2,x_3,...),|y_1,y_2,y_3,...)) = 2^{-n},$
where $n=\min\{i\,|\,x_i\neq y_i\}$,  if $x\neq y$. It is a measurable space when equipped with the Borel $\sigma-$algebra $\mathcal{B}$.
For any $n\geq 1$, and any fixed symbols $b_1,...,b_n$ in $\{1,2,...,d\}$, we  define the cylinder set
$|b_1,b_2,...,b_n]= \{|x_1,x_2,x_3,...)\in X \,| \, x_1=b_1,...,x_n=b_n \}$.
A Borel probability $\mu$ on $\Omega$ is called shift-invariant if it satisfies
$\mu(|b_1,b_2,...,b_n]) = \sum_{i=1}^{d}\mu(|i,b_1,...,b_n])$
for any cylinder set. Finally, the Kolmogorov-Sinai entropy of a shift-invariant Borel probability $\mu$ is given by
\begin{equation}\label{kolmogorov}
	h(\mu)=\lim_{n\to\infty}-\frac{1}{n}\sum_{i_1,...,i_n}\mu(|i_1,...,i_n])\log(\mu(|i_1,...,i_n]) ).
\end{equation}


In Thermodynamic Formalism (see \cite{PP}, \cite{Walters}) it is  quite common to consider the concept of pressure for a Lipschitz potential $\phi:\Omega\to\mathbb{R}$, where $\Omega=\{1,2,...,d\}^{\mathbb{N}}$. We say that a shift-invariant probability $\mu_{\phi}$ is the equilibrium probability for the Lipschitz function $\phi:\Omega\to\mathbb{R}$, if
\[ P(\phi):= \sup_{\mu\,\text{shift-invariant}} [\int \phi\, d\mu + h(\mu)] = \int \phi\, d\mu_\phi + h(\mu_\phi).\]
The number $P(\phi)$ is called the pressure of the potential $\phi$.

In ergodic theory for symbolic dynamics appears also the concept of specific information gain which can be used to introduce the entropy production for equilibrium probabilities (see \cite{Jiang} for an introduction to these concepts in a setting compatible with the present work).  
If $\mu_\phi$ is the equilibrium probability for the Lipschitz function $\phi$ and if $\mu$ is shift-invariant, then, the \textbf{specific information gain} of $\mu$ with respect to $\mu_\phi$ is given by
\begin{equation}\label{eq2}
	h(\mu,\mu_\phi) :=  \lim_{n} \frac{1}{n} \sum_{|i_1,...,i_n]} \mu(|i_1,...,i_n])\log\left(\frac{\mu(|i_1,...,i_n])}{\mu_\phi(|i_1,...,i_n])}\right).
\end{equation}
Furthermore, from Proposition 1 in \cite{Jiang} (see also \cite{Cha1}), we get
\begin{equation} \label{eq3}
	h(\mu,\mu_\phi)  =\underbrace{[\int \phi\,d\mu_\phi +h(\mu_\phi)]}_{P(\phi)}-[\int \phi\,d\mu +h(\mu)] .
\end{equation}

In section \ref{secbasic} we exhibit some analogies between equations \eqref{eq3} and \eqref{IG}.  Comparing the equations \eqref{kolmogorov} and \eqref{eq2} it is natural to interpret the specific information gain as a relative entropy. Furthermore, in \cite{Cha1} the value $h(\mu,\mu_\phi) $ is characterized by a variant of the Shannon-McMillan-Breiman theorem.
Indeed, from a  result on section 3.2 of \cite{Cha1} we get the following:  consider an ergodic probability $\mu$ on $\Omega$, and for a given Lipschitz function $\phi:\Omega\to\mathbb{R}$, consider the corresponding equilibrium probability $\mu_\phi$. Then,
for $\mu$ almost every point $x=|x_1,x_2,x_3...)\in \Omega$,
\begin{equation} \label{Chacha1} \lim_{n \to \infty} \frac{ 1 }{ n}  \log \left(\frac{\mu  \,(| x_1,x_2,...,x_n] )} {\mu_\phi \,(\,| x_1,x_2,...,x_n ]\,) }\right) = h(\mu,\mu_\phi) .
\end{equation}

An interpretation of this expression  in the sense of the Statistical Mechanics of non equilibrium is the following:
the observed  system $ \mu_\phi$ is the equilibrium probability for the Lipschitz function $\phi$, then,
given a random point $x\in \Omega $,  its time $n-1$ orbit  $\{x, \sigma(x),..., \sigma^{n-1} (x)\,\} $ describes the dynamical evolution of the system under consideration. For each $n\in \mathbb{N}$, let  $\displaystyle{\nu_n^x=\frac{1}{n}\, (\,\delta_x\,+\, \delta_{\sigma(x)}\, + ... +  \delta_{\sigma^{n-1} (x)}\,)}$
the associated probability to $x$ at time $n$ (the empirical measure). Then, from Birkhoff's ergodic theorem,  for $\mu_\phi$ a.e. $x$, we get that $\nu_n^x \to \mu_\phi$, as $n \to \infty.$
Denote by $\mu$ another ergodic probability (which is not the equilibrium for $\phi$).
Then, for $\mu$ a.e. $x$, as $n \to \infty$, we get (in the sense of \eqref{Chacha1}) 
$$\frac{\mu \,(\, |x_1,x_2,...,x_n ] \,)}{\mu_\phi \,(\,|x_1,x_2,...,x_n ]\,) }\,\sim\, e^{  n \,h(\mu,\mu_\phi)}.$$
Therefore, the value $\, h(\mu,\mu_\phi)$ quantifies the asymptotic exponential rate which describes how the dynamical time evolution of the system discriminates between $\mu_\phi$ and $\mu$, when $n \to \infty$.

\medskip
 We will present now the concept of entropy production for equilibrium probabilities on $\{1,...,d\}^{\mathbb{N}}$ and its relations with the specific information gain.  Remember that -  for the sake of notation see  (\ref{jai}) - we denote $\Omega$ by $\Omega^{+}$. The elements of  $\Omega^{+}$ are denoted by $x=|\,x_1,x_2,...).$
Consider  the space $\Omega^-=\{1,2,...,d\}^{\mathbb{N}}$ where
any point in the space  $\Omega^-$ will be written in the form $y=(...,y_{3},y_{2},y_1|$. In this way any point in $\hat{\Omega}:=\Omega^-\times \Omega^+$ will be written in the form $(...,y_{3},y_{2},y_1|x_1,x_2,x_3,...)=(y\,|\, x)$.

We consider on $\hat{\Omega}$ the shift map $\hat{\sigma}$ given by
\begin{equation} \label{ret}
	\hat{\sigma}((...,y_{3},y_{2},y_1|x_1,x_2,x_3,...)) = (...,y_{3},y_{2},y_1,x_1|x_2,x_3,...).
\end{equation}
The natural restriction of  $\hat{\sigma}$ over $\Omega=\Omega^{+}$ is the shift map $\sigma$. The natural restriction of $\hat{\sigma}^{-1}$ over $\Omega^-$ is denoted by $\sigma^-$. Observe that $(\Omega^-,\sigma^-)$  can be identified with $(\Omega^{+},\sigma)$, via  the conjugation $\theta: \Omega^-\to \Omega^{+}=\Omega$, which is  given by
\begin{equation} \label{ket} \theta((...,z_{3},z_{2},z_1|) = |z_1,z_2,z_3,...). \end{equation}

Any $\sigma$-invariant  probability $\mu$ on $\Omega^+$ can be extended (uniquely) to a $\hat{\sigma}$-invariant probability $\hat{\mu}$ on $\Omega^-\times \Omega$. The restriction of $\hat{\mu}$ to $\Omega^{-}$, denoted by $\mu^{-}$, is  $\sigma^-$-invariant. By identifying $(\Omega^{-},\sigma^-)$ with $(\Omega,\sigma)$, via the conjugation $\theta$ and denoting by $\theta_*\mu^-$ the push forward of $\mu^-$, we get
\begin{equation}\label{eq1}
	\theta_*\mu^-(|a_1,a_2....a_m]) = \mu(|a_m,...,a_{2},a_1]).
\end{equation}

Finally, the \textbf{entropy production} of an equilibrium probability $\mu$ on $\Omega$ is defined by the specific information gain $e_p(\mu) := h(\mu\,,\,\theta_*\mu^-)$, that is,
\[e_p(\mu) = h(\mu\,,\,\theta_*\mu^-)= \lim_{n} \frac{1}{n} \sum_{|a_1,...,a_n]} \mu(|a_1,...,a_n])\log\left(\frac{\mu(|a_1,...,a_n])}{\mu(|a_n,...,a_{2},a_1])}\right).\]

From now on we consider more general spaces (measurable spaces or compact metric spaces) instead of finite sets or finite alphabet.

In Information Theory, for a measurable space $X$, the Shannon entropy is extended by the Kullback-Leibler divergence (see \cite{KL}) given by
\[D_{KL}(P|\nu)= \left\{\begin{array}{l}\int \log(\frac{dP}{d\nu})\,dP\,\,\,\,\text{if}\,\, P\ll \nu \\\\ +\infty, \,\,\,\, \text{otherwise}\end{array}\right.,  \]
where $\nu$ can be interpreted as an a priori probability on $X$, $P$ is another probability on $X$ and $P\ll\nu$ means that $P$ is absolutely continuous with respect to $\nu$.

The K-L divergence is also used to extend for measurable spaces the information gain or mutual information. If $\pi$ is absolutely continuous with respect to $P\times Q$, then, denoting by  $\frac{d\pi}{dPdQ}$ the Radon-Nikodyn derivative, the mutual information can be expressed in terms of
\begin{equation}\label{KLmutual}
	D_{KL}(\pi\,|\,P\times Q) = \int \log \left(\frac{d\pi}{dPdQ}(x,y)\right)\,d\pi(x,y).
\end{equation}

From another point of view, in the TFCA model studied in \cite{LMMS}, which considers a symbolic dynamic over an alphabet given by a compact metric space $M$ (instead of a finite or enumerable set), it was proposed to consider a relative entropy given by\small
\[
h^{\nu}(\mu) := -\sup\{\int c(|x_1,x_2,...))\,d\mu(|x_1,x_2,...))\,|\, \int e^{c(|a,w))}\,d\nu(a) = 1\,\forall w\in M^{\mathbb{N}}\},
\]\normalsize
where $\nu$ is an a priori probability on $M$, $\mu$ is a shift-invariant (stationary) probability on $\Omega:=M^{\mathbb{N}}=\{|x_1,x_2,x_3,...)|x_i \in M \,\,\forall i \in \mathbb{N}\}$ and the functions $c:M^{\mathbb{N}}\to\mathbb{R}$ are necessarily Lipschitz. Variations of this expression appear in $\cite{LMMS2}$,  $\cite{MO}$ and more recently in $\cite{LM}$. 

In \cite{ACR} it was proved that $h^{\nu}$ coincides with the specific entropy in Statistical Mechanics, which is related to the $D_{KL}$.  In the present work we propose to rewrite $h^{\nu}$ in terms of a variational characterization of \eqref{KLmutual}  which assures that $h^\nu(\mu)$ is related with $D_{KL}$ in a more direct way than \cite{ACR}. Precisely, if $P$ is a probability on $X$ and $\pi$ is a probability on $X\times Y$ with $y-$marginal $Q$, then from Theorem \ref{theoKL} we obtain that \footnotesize
\begin{equation}\label{DKLsup} D_{KL}(\pi \,|\, \nu\times Q) = \sup\left\{\int c(x,y)\,d\pi(x,y)\,|\, \int e^{c(x,y)}d\nu(x) =1\,\forall y,\, c\, \in \mathcal{F}(\pi)\right\},\end{equation} \normalsize
where $c\in \mathcal{F}(\pi)$ if $c$ is a measurable function such that $\int c\, d\pi$ is well defined (it is not $+\infty -\infty$). Furthermore, in Theorem \ref{entropylip} we prove that for compact metric spaces $X$ and $Y$ the above supremum can be taken over Lipschitz functions.

We notice that taking $X:=M$, $Y:=M^{\{2,3,4,5,...\}}$, and identifying $\Omega$ with $X\times Y$ by the rule
\begin{equation}  \Omega \ni |x_1,x_2,x_3,x_4,...) \to (x_1,\,|x_2,x_3,x_4,...)) \in X\times Y,  \end{equation}	
then, the entropy $h^{\nu}$ can be rewritten as
\[h^{\nu}(\mu) = -\sup\{\int c(x,y)\,d\mu(x,y)\,|\, \int e^{c(x,y)}\,d\nu(x) = 1\,\forall y\in Y\},\]
where the supremum is taken over Lipschitz functions.
It follows from \eqref{DKLsup} that the entropy proposed in \cite{LMMS} (and \cite{MO}) can be rewritten in terms of $D_{KL}$. We elaborate more about this issue for the case of the TFCA model in section \ref{xy}.

In section \ref{secIFS} we introduce the concept of information gain with respect to a probability kernel.

\begin{definition}\label{deftransverse}
Let $X$ and $Y$ be measurable spaces. We will call of a \textbf{probability kernel} any family $\hat{\nu}=\{\hat{\nu}^y\,|\,y\in Y\}$ of probabilities on $X\times Y$, such that,
	
	1)  $\forall \,y \in Y$, we have $\hat{\nu}^y(X_y)=1$, where $X_y=\{(x,y)\,|\,x\in X\}$,
	
	2) $\forall A \subset X\times Y$ measurable, we have that $y\to \nu^y (A)$ is measurable.
\end{definition}

If $\hat{\nu}$ is a probability kernel and $Q$ is a probability on $Y$, then we can define a probability $\pi=\hat{\nu}\,dQ$ on $X\times Y$ by $\pi(A) = \int \hat{\nu}^y(A)\,dQ(y)$. It means
\begin{equation}\label{transdisint}
\int f(x,y)\,d\pi(x,y):= \int f(x,y) \hat{\nu}^{y}(dx)dQ(y).
\end{equation} 
The  right-hand side of the above expression  can be seen as a Rokhlin's disintegration of $\pi$. 

Following \cite{LM} we consider for the present setting the definition of entropy described below.

\begin{definition} Let $X$ and $Y$ be measurable spaces. We define the entropy of any probability $\pi$ on $X\times Y$ relative to the probability kernel $\hat{\nu}$ as
	\[H^{\hat{\nu}}(\pi) = -\sup\{\int c(x,y)\,d\pi(x,y)\,|\, \int e^{c(x,y)}\hat{\nu}^y(dx) =1\,\forall y,\, c \in \mathcal{F}(\pi)\}.\]
\end{definition}

Finally, we will introduce and study the following meaning of information gain, which is able to extend all the different notions of information gain considered in this paper.

\begin{definition} Let $X$ and $Y$ be measurable spaces. We define the \textbf{information gain} of a probability $\pi$ on $X\times Y$ relative to the probability kernel $\hat{\nu}$, by
	\[IG(\pi,\hat{\nu}) = -H^{\hat{\nu}}(\pi).\]
\end{definition}

If $\pi$ has a $y$-marginal $Q$ then from Theorem \ref{theoKL} of section \ref{secIFS} we have that 
\[IG(\pi,\hat{\nu}) = D_{KL}(\pi\,|\, \hat{\nu}\,dQ).\]

Following \cite{LM}, it is possible to remark that there are natural extensions of the above concepts if we replace $X\times Y$ by a measurable space $M$ with a measurable partition (which induces an equivalence relation) and probability kernels by general transverse functions. On the other hand, the above information gain is related with the generalized conditional relative entropy (see chap. 5 in \cite{Gray}) in the following sense:  If $\pi$ has $y-$marginal $Q$ and $\pi_0$ has a disintegration $\pi_0 =\hat{\nu}\,d\tilde{Q} $ then the  conditional relative entropy of $\pi$ with respect to $\pi_0$ is given by $D_{KL}(\pi| \hat{\nu}\,dQ)$ and therefore its value coincides with $IG(\pi,\hat{\nu})$ above defined. 

In addition to being connected with the previous work \cite{LM}, we remark that there are at least two natural reasons for our preference of the above approach using probability kernels instead of a probability $\pi_0$. The first one is because the conditional relative entropy, as above defined, does not consider $\pi_0$ totally, but only $\hat{\nu}$, while the $y-$marginal $\tilde{Q}$ of $\pi_0$ is replaced by $Q$. So it is not necessary to compute a disintegration (or a regular conditional probability measure) $\hat{\nu}$ for $\pi_0$, but just to consider a priori such probability kernel $\hat{\nu}$ instead of $\pi_0$. In this case, it is not necessary to impose more restrictions on the spaces which would be necessary in order to get a disintegration. The second one is that for a fixed probability $\pi_0$ the regular conditional probability measure $\hat{\nu}$ is in general not unique. If $\hat{\nu}$ and $\hat{\mu}$ are different probability kernels satisfying 
\[\pi_0 = \hat{\nu}\,d\tilde{Q} = \hat{\mu}\,d\tilde{Q}\] then the conditional relative entropy may not be well defined and more assumptions are required, as for example $Q\ll \tilde{Q}$ . In section \ref{seccompact} we consider compact metrical spaces $X$ and $Y$ and show that, under some assumptions on $\pi_0$, an information gain (or, conditional relative entropy) $IG(\pi,\pi_0)$ can be naturally introduced.

The above generalized information gain will be used  in section \ref{xy}  to introduces the concept of  information gain in the TFCA model. Finally, we will be able to introduce the definition of entropy production in the TFCA model (see section \ref{secEP}). In our reasoning, it will be natural to use as a tool the concept of involution kernel (for references about the involution kernel with setting compatible with the present paper see \cite{BLT}, \cite{LMMS} and \cite{LOT}).
We will show (see Corollary \ref{variational21}) that in the case the potential is symmetric the associated equilibrium probability has zero entropy production.

In \cite{LR} the authors analyze the change of the  KL-divergence for Gibbs probabilities under the action of the dual of the Ruelle operator.

Results related to the role of the entropy production (the fluctuation theorem and the detailed balance condition) in problems in Physics and Dynamics can be found in \cite{GC}, \cite{Jiang}, \cite{Maes}, \cite{Rue} and  \cite{Ben}.
 A  concrete example of a system where the entropy production plays an important role  is presented in \cite{Croo}: a classical gas confined in a
cylinder by a movable piston (see the first page of \cite{Croo}).

We would like to thank L. Cioletti for helpful comments during the writing of  this paper.

\section{Relations between the different concepts of information gain}\label{secbasic}

In this section we propose to explain a relation between the information gain given by \eqref{IG} and the specific information gain given by \eqref{eq2} and \eqref{eq3}.  In Thermodynamic Formalism a Lipschitz potential $\phi: \Omega \to\mathbb{R}$ is called \textbf{normalized} if $\sum_{x_1}e^{\phi(|x_1,x_2,x_3,...))} =1,\, \forall \,x_2,x_3,... \in \{1,...,d\}$. In this case $P(\phi)=0$ and furthermore
\[  e^{\phi(|x_1,x_2,x_3,...))} = \lim_{n\to\infty} \frac{\mu_{\phi}(|x_1,x_2,...,x_n])}{\sum_i \mu_\phi(|i,x_2,...,x_n])}\]
(see  \cite{PP}, cor. 3.2.2). We will call, for any shift-invariant probability $\mu$, \textbf{Jacobian} of $\mu$ the function
$$J^{\mu}(|x_1,x_2,...)) :=  \lim_{n\to\infty} \frac{\mu(|x_1,x_2,...,x_n])}{\sum_i \mu(|i,x_2,...,x_n])}=\lim_{n\to\infty} \frac{\mu(|x_1,x_2,...,x_n])}{ \mu(|x_2,...,x_n])},$$
which is defined $\mu-$a.e\footnote{our abstract definition corresponds to the inverse of the usual Jacobian $T^\prime$ for the action of a locally invertible map $T$  and the Lebesgue measure.} . In this way, for a normalized potential $\phi$, we have that $\log(J^{\mu_\phi}) = \phi$, and \eqref{eq3} can be rewritten as
\begin{equation}\label{eq3new}
h(\mu,\mu_\phi)  =-[\int \log(J^{\mu_\phi})\,d\mu +h(\mu)].
\end{equation}
We also remark that from Lemma 7 in \cite{LMMS2} the Kolmogorov-Sinai entropy satisfies

\begin{equation}\label{sbm}
h(\mu) = -\sup\left\{\int c\,d\mu \,|\, \begin{array}{c} c\,\,\text{is Lipschitz and}\\ \sum_{x_1}e^{c(|x_1,x_2,x_3...))} = 1\\ \,\forall\,x_2,x_3,...\in\{1,...,d\}\end{array}\right\}.
\end{equation}
In order to explain the relations between $h(\mu,\mu_\phi)$ and the information gain given by $\eqref{IG}$ we need also to extend $\eqref{IG}$. For a probability $\pi$ on the finite set $X\times Y$, we will call $J^{\pi}(x,y) := \frac{\pi_{x,y}}{\sum_x \pi_{x,y}}$ the \textbf{Jacobian} of the probability $\pi$ (which is defined $\pi$-a.e.). Then, we have that $H(\pi)$ given in \eqref{IG} satisfies $$H(\pi) = -\sum_{x=1}^{d}\sum_{y=1}^{r} {\pi_{x,y}}\log(J^{\pi}(x,y))=-\int \log(J^{\pi})\,d\pi.$$

In proposition \ref{Hfinito} of Appendix section \ref{appendix2} we will prove (in a similar way as in chap. 3 in \cite{M}) that
\begin{equation}\label{eqHfinito}	H(\pi) = - \sup \{ \sum_{x,y} f(x,y)\pi_{x,y}\,|\, \sum_{x\in X} e^{f(x,y)} = 1, \, \forall y\}.
\end{equation}

For any given probability $P$ on $X=\{1,...,d\}$ and any given probability 	$\tilde{Q}=(\tilde{q}_1,...,\tilde{q}_r)$ on $Y=\{1,...,r\}$, with $\tilde{q}_i>0,\, \forall i$,  consider the product measure $\pi_0 =P\times\tilde{Q}$ on  $ X\times Y $. Then,\newline
1. $\displaystyle{J^{\pi_0}(x,y)=\frac{p_x\tilde{q}_y}{\sum_xp_x\tilde{q}_y}=\frac{p_x\tilde{q}_y}{\tilde{q}_y}= p_x}$,\newline
2. $ \displaystyle{S(P)=-\sum_{x,y} {(\pi_0)_{x,y}}\log(p_x)= -\sum_{x,y} ({\pi_0})_{x,y}\log(J^{\pi_0}(x,y)) =H(\pi_0)}$,\newline
3. If $\pi $ is any probability on $X\times Y$ with $x$-marginal $P$, then,
\begin{align*}
IG(\pi, P) &= S(P) - H(\pi) \stackrel{2.}{=}   H(\pi_0) - H(\pi)\\&= - [\int \log(J^{\pi_0})\,d\pi_0 +H(\pi)]\\&=- [\int \log(J^{\pi_0})\,d\pi +H(\pi)],
\end{align*}
where the last equality is satisfied because $J^{\pi_0}(x,y)\stackrel{1.}{=}p_x$ depends only on the first coordinate, and the $x$-marginal of both probabilities $\pi$ and $\pi_0$ is the probability $P$.

This allows us to extend the definition of information gain \eqref{IG} in the following way:

\begin{definition}
Let  $\pi_0, \pi$ be probabilities on $X\times Y$, such that $(\pi_0)_{x,y}>0,\,\forall\,(x,y)\in X\times Y$. We define the information gain of $\pi$ with respect to $\pi_0$ by
\begin{equation}\label{IG2}
IG(\pi,\pi_0) =  - [\int \log( J^{\pi_0})\,d\pi +H(\pi)].
\end{equation}
\end{definition}

The expression of the information gain $IG(\pi,\pi_0)$ and the expression of the specific information gain $h(\mu,\mu_\phi)$ given in \eqref{eq3new} are similar. Furthermore, the Jacobians
and both variational characterizations of $h(\mu)$ and $H(\pi)$ given in \eqref{sbm} and \eqref{eqHfinito} are alike.


We believe that the next remark can help the reader in understanding why
the introduction of probability kernels is natural to replace finite
sets by measurable sets (in the study of Information gain).

\begin{remark}\label{remarkidea}
In the right hand side of above expression \eqref{IG2} does not appear $\pi_0$ but only $J^{\pi_0}$. If $\tilde{Q}$ is the $y$ marginal of $\pi_0$ then, by definition of $J^{\pi_0}$, for any function $f$,
\[\sum_{x,y} f(x,y) \pi_0(x,y) = \sum_{x,y} f(x,y)J^{\pi_0}(x,y)\tilde{Q}(y).\]
Furthermore, for each fixed $y$ we have that $\sum_x J^{\pi_0}(x,y) =1$. Therefore, for each fixed $y$, we can interpret $J^{\pi_0}$ as a probability in $X\times\{y\}$. In this way $J^{\pi_0}$ is a probability kernel in the sense of Definition \ref{deftransverse}. A similar remark is true for \eqref{eq3new}.		
\end{remark}

\section{Information Gain and probability kernels}\label{secIFS}

Our purpose in this  section is to extend the definition of information gain  $IG(\pi,\pi_0)$, given by \eqref{IG2}, for the case when $X$ and $Y$ are measurable spaces.  As we will see, a natural way of to extends \eqref{IG2} is by considering probability kernels and the notion of entropy given in \cite{LM}. This entropy is an extension of that previously introduced in \cite{LMMS} and \cite{MO} for compact spaces using an a priori probability. In \cite{MO} an entropy has been introduced for holonomic probabilities associated with iterated function systems (IFS), but we point out that the expression of the entropy in \cite{MO} does not use such structures. It may seem
 surprising, but it is related, by a variational principle, with the spectral radius of a transfer operator which is defined from the IFS. As we will see below the entropy considered in this section does not consider any dynamics.

From now on we consider $\sigma-$algebras $\mathcal{A}$ on $X$ and $\mathcal{B}$ on $Y$ and the product $\sigma-$algebra on $X\times Y$.
If $c:X\times Y\to\mathbb{R}$ is measurable then, for any fixed $y\in Y$, the function  $c_y(x):=c(x,y)$ defined on $X$ is measurable (see \cite{CK} theorem 6.7).

In order to make an identification with the setting of \cite{LM} we consider in the space $X\times Y$ the equivalence relation $(x_1,y_1)\sim(x_2,y_2)$, if and only if, $y_1=y_2$. So the equivalence classes are the horizontal lines of $X\times Y$. The so called transverse functions in \cite{LM} corresponds to probability kernels in the present setting.

 \begin{definition}
 	We will call of a \textbf{probability kernel} any family $\hat{\nu}=\{\hat{\nu}^y\,|\,y\in Y\}$ of probabilities on $X\times Y$, such that,
 	
 	1)  $\forall \,y \in Y$, we have $\hat{\nu}^y(X_y)=1$, where $X_y=\{(x,y)\,|\,x\in X\}$,
 	
 	2) $\forall A \subset X\times Y$ measurable, we have that $y\to \nu^y (A)$ is measurable.
 \end{definition}

If $X$ and $Y$ are metric spaces, as considered in \cite{LM}, then condition 1) is equivalent to say that the probability $\nu^y$ has support on $X_y$.  Another equivalent way of defining a probability kernel is as a family of probabilities $\hat{\nu}^y$ on $X$ such that for any measurable set $B\subset X$ we have that $y\to \nu^y (B)$ is measurable (to prove this statement, just adapt the reasoning of Theorem 6.4 in \cite{CK} to the current setting).

The next definition was taken from the reasoning of  \cite{LM}.

 \begin{definition} We define the entropy of any probability $\pi$ on $X\times Y$ relative to the probability kernel $\hat{\nu}$ as
 	\[H^{\hat{\nu}}(\pi) = -\sup\{\int c(x,y)\,d\pi(x,y)\,|\, \int e^{c(x,y)}\hat{\nu}^y(dx) =1\,\forall y,\, c\in \mathcal{F}(\pi)\},\]
 	where $\mathcal{F}(\pi)$ is the set of measurable functions with a well-defined integral with respect to $\pi$.
 \end{definition}

A well-defined integral, in the above definition, means that it is not $+\infty -\infty$. It follows from Lemma \ref{lemagap} below that we can take the above supremum over functions $c$ which are bounded below. Such functions belongs to $\mathcal{F}(\pi)$, even though we may have $\int c\,d\pi =+\infty$.

 Usually, we also fix  a probability $\nu$ on $X$ satisfying $\operatorname{supp}(\nu)=X$, which we call an {\it a priori} probability on $X$.  Given an a priori probability $\nu$ on $X$ and considering the identification of $X$ and $X_y$, we can consider the a priori probability kernel $\hat{\nu}$ given by $\hat{\nu}^{y}(dx) = \nu(dx)$ (for condition 2. see Theorem 6.4 in \cite{CK}). In this case we write $\hat{\nu} \equiv \nu$, and we denote $H^{\hat{\nu}}(\pi)$ also by $H^{{\nu}}(\pi)$, which will be given by
  \[H^{{\nu}}(\pi) = -\sup\{\int c(x,y)\,d\pi(x,y)\,|\, \int e^{c(x,y)}d\nu(x) =1\,\forall y,\, c\in\mathcal{F}(\pi)\}.\]

\begin{definition} We say that a measurable function $c:X\times Y \to \mathbb{R}$ is $\hat{\nu}$-normalized if
 \[\int e^{c(x,y)}\hat{\nu}^y(dx) =1\,,\forall y \in Y.\]
 If $\nu$ is an a priori probability on $X$, we say that $c:X\to\mathbb{R}$ is $\nu-$normalized, if it is measurable and $\int e^{c}d\nu=1$.
\end{definition}

\begin{example} If $\nu$ is an a priori probability on $X$ and $c:X\times Y \to \mathbb{R}$ is ${\nu}$-normalized, that is, 	\[\int e^{c(x,y)}d\nu(x) =1\,,\forall y \in Y,\]
then defining, for each $y$, the probability $\hat{\nu}^y$ on $X_y$ by $\hat{\nu}^y(dx) := e^{c(x,y)}d\nu(x)$, we get that $\hat{\nu}$ is a probability kernel. It corresponds to the case where all the probabilities $\hat{\nu}^y$ are densities for the same probability $\nu$. More generally, if $\hat{\nu}$ is a probability kernel and $c$ is $\hat{\nu}-$normalized, then $e^{c(x,y)}\hat{\nu}^{y}(dx)$ is a probability kernel too. If $Q$ is a probability on $Y$ we get also a probability $\pi:=\hat{\nu}\,dQ$ on $X\times Y$ by \eqref{transdisint} and the right hand side is a disintegration of $\pi$ with respect to the horizontal lines of $X\times Y$.  If $\hat{\nu}\equiv \nu$ we get $\hat{\nu}dQ = d\nu dQ$ is a product measure.

\end{example}

 \bigskip

The function $c=0$ is $\hat{\nu}-$normalized and therefore $H^{\hat{\nu}}(\pi)\leq 0$.
If  $\tilde{{\nu}}$ is a finite measure on $X$ satisfying $\tilde{{\nu}}(X)= d$ and $d{\hat{\nu}} \equiv \frac{1}{d}d\tilde{{\nu}}$, then $H^{\tilde{{\nu}}}(\pi) = H^{\hat{\nu}}(\pi) + \log(d)$, where $H^{\tilde{{\nu}}}$ is defined in a similar way. Now, taking  $X$ and $Y$ as finite sets, ${\tilde{\nu}}$ as the counting measure on $X$ and applying equation \eqref{eqHfinito},  we came to the conclusion that such definition of entropy is a natural extension of the definition of $H(\pi)$.

\bigskip

If $P$ is a probability on $X$ we also define
\[S^{\nu}(P) := - \sup\left\{ \int c(x)\,dP(x)\,|\, \int e^{c(x)}\,d\nu(x)=1 ,\, \,\text{where}\, c \in\mathcal{F}(P)\right\}.\]

We start by proving the next theorem which shows that the above definitions provide variational characterizations of the Kullback-Leibler divergence. It also shows that $-H^{\hat{\nu}}$ is equivalent to generalized conditional relative entropy (see chap. 5 in \cite{Gray}) as explained in the introduction section.

\begin{theorem}\label{theoKL} Let $P$ and $\nu$ be probabilities on $X$, $\pi$ be a probability on $X\times Y$ with $y-$marginal $Q$ and $\hat{\nu}$ be a probability kernel. Then,
	\[S^{\nu}(P) = - D_{KL}(P\,|\,\nu) \,\,\,\,and\,\,\,\, H^{\hat{\nu}}(\pi) = - D_{KL}(\pi \,|\, \hat{\nu} \,d Q).\]
	Consequently, if $\hat{\nu}\equiv \nu$, we have  $H^{{\nu}}(\pi) = - D_{KL}(\pi \,|\, \nu\times Q).$
\end{theorem}
 The proof will be divided into several lemmas (the last one is Lemma \ref{notabscont} which will finish the proof).
 
 \begin{remark} It is known that (see for example chap.5 of \cite{Gray})
 \small	\[D_{KL}(P \,|\, \nu) = \sup\left\{ \int c\,dP -\log( \int e^{c}\,d\nu),\, \,\text{where}\, c \, \in\mathcal{F}(P)\,and\, \int e^{c}\,d\nu<\infty \right\}. \]\normalsize
  It follows that 	
 \[D_{KL}(P \,|\, {\nu} ) = -S^{\nu}(P)\,\, \text{and} \,\, D_{KL}(\pi \,|\, \hat{\nu} \,d Q) \geq -H^{\hat{\nu}}(\pi).\]  Anyway we provide a complete proof.
 	\end{remark} 
 
 \begin{definition} We say that a measurable function $J:X\times Y \to [0,+\infty)$ is a $\hat{\nu}-Jacobian$, if   $\int J(x,y)\,\hat{\nu}^y(dx) =1,\, \forall y\in Y$. \newline
Given an a priori probability $\nu$ on $X$, we say that a measurable function $J:X\to[0,+\infty)$ is a ${\nu}-$Jacobian if $\int J d{\nu}=1$.	
 \end{definition}

 If $c$ is $\hat{\nu}-$normalized, then $J=e^{c}$ is a $\hat{\nu}$-Jacobian. On the other hand, if $J$ is a $\hat{\nu}-$Jacobian and it does not assume the value zero, then $c=\log(J)$ is $\hat{\nu}-$normalized.

 \begin{lemma}\label{lemagap}
\[ H^{\hat{\nu}}(\pi) = - \sup\left\{ \int \log(J(x,y))\,d\pi(x,y)\,|\, J\, \,\text{is a}\, \,\,\hat{\nu}\text{-Jacobian} \right\}.\]
Furthermore, we can also take the supremum over positive Jacobians $J$ satisfying $\inf\{J(x,y)|x\in X, y\in Y\}>0$. 
 \end{lemma} 	
 \begin{proof}
 	If $c$ is $\hat{\nu}-$normalized then $J=e^{c}$ is a $\hat{\nu}$-Jacobian.	 It follows that
 	\[\sup\left\{ \int \log(J(x,y))\,d\pi(x,y)\,|\, J\, \text{is a}\, \,\hat{\nu}\text{-Jacobian} \right\}\]
 	\[\geq \sup\left\{ \int c(x,y)\,d\pi(x,y)\,|\, c\,\, \text{is}\, \,\hat{\nu}\text{-normalized} \right\}.\]
 	On the other hand, for any fixed $\hat{\nu}$-Jacobian $J$, let $J_n=\frac{J+\frac{1}{n}}{1+\frac{1}{n}}$. As $J\geq 0$ we have $J_n \geq \frac{1}{n+1}$. The function $J_n$ is also a $\hat{\nu}-$Jacobian. Furthermore,
 	\[\int \log(J)\,d\pi \leq \liminf_n \int \log (J+\frac{1}{n})d\pi = \liminf_n \int \log(J_n)\,d\pi.\]
 	As the function $c_n=\log(J_n)$ is $\hat{\nu}-$normalized and bounded below, this ends the proof.
 \end{proof}

\begin{lemma} Let $P$ be a probability on $X$ and $\pi$ be a probability on $X\times Y$, with $x-$marginal $P$. If $P$ is not absolutely continuous with respect to the a priori probability $\nu$, then $S^{\nu}(P)=-\infty$ and $H^{\nu}(\pi) = -\infty$.
\end{lemma}
\begin{proof}
If $P$ is not absolutely continuous 	
	with respect to $\nu$ then there exists a measurable set $A$, such that, $\nu(A)=0$ and $P(A)>0$. For each $\beta>0$, let $c_\beta:X\to\mathbb{R}$ be the measurable function defined as
	\[ c_{\beta}(x) = \left\{\begin{array}{cl} 0 & \text{if}\, x\in X-A\\ \beta & \text{if}\, x\in A\end{array} \right. .\]
	Then, we have $\int e^{c_\beta(x)}\,d\nu(x) = 1$ and $\int c_\beta(x)\, d\pi(x,y) = \int c_\beta(x)\, dP(x) = \beta P(A)$. 	As $\beta$ is arbitrary, we can take $\beta \to +\infty$, and then we get that $S^{\nu}(P) =-\infty$ and also that $H^{\nu}(\pi) = -\infty$.
\end{proof}

As usual, we use the notation $\mu \ll \nu$ to denote that $\mu$ is absolutely continuous with respect to $\nu$. If $P\ll \nu$ we denote by $\frac{dP}{d\nu}$ the Radon-Nikodyn derivative of $P$ with respect to $\nu$, which is a  measurable function.

Observe that $J_0:=\frac{dP}{d\nu}$ is a $\nu-$Jacobian.
Let $X_0=\{x\in X\,|\, J_0(x)>0\}$. Given any measurable and bounded function $f:X\to\mathbb{R}$, we have
\[\int f\,dP = \int f \cdot J_0\,d\nu = \int f\cdot I_{X_0}\cdot J_0\,d\nu = \int_{X_0}f\,dP.\]
It follows that $P(X_0)=1$ and
$\int f(x) \, dP(x) = \int_{X_0} f(x)\, dP(x)  $,  for any measurable function $f$.

Furthermore the integral $\int \log(J_0) dP$ is well defined (it can be $+\infty$) because $\int \log(J_0) dP = \int \log(\frac{dP}{d\nu})\frac{d P}{d\nu} d\nu $ and the function $x\log(x)$ is bounded below.  

\bigskip

The next result shows that $-S^{\nu}(P)$ is the  Kullback-Leibler divergence of $P$ with respect to $\nu$.

\begin{lemma}\label{kullback} Let $P$ be a probability on $X$, such that, $P\ll \nu$.  Then, $$S^{\nu}(P) = -D_{KL}(P\,|\,\nu)=-\int \log(\frac{dP}{d\nu})\,dP.$$
\end{lemma}
\begin{proof} Let $J_0:=\frac{dP}{d\nu}$ and $X_0:=\{x\in X\,|\, J_0(x)>0\}$. We claim that \[\int_{X_0}\log(J_0)\, dP =  \sup\left\{ \int_{X_0} \log(J(x))\,dP(x)\,|\,  J\,\,\text{is a $\nu-$Jacobian}\right\}.\]
Indeed, from Lemma \ref{lemagap} we can consider a $\nu-$Jacobian $J:X\to(0,+\infty)$ such that $\inf_{x,y} J(x,y)>0$. By applying the Jensen's inequality we have
\[\int_{X_0} \log\left(\frac{J}{J_0}\right)\,dP \leq \log \int_{X_0}\frac{J}{J_0}\,dP = \log \int_{X_0} J \,d\nu \leq    \log \int J\,d\nu =0.\]
This shows that
\[ \int_{X_0} \log(J)\, dP \leq \int_{X_0} \log(J_0)\, dP .\]
\end{proof}	

A similar result for $\pi$ will be  given by the next result.

\begin{lemma}\label{pijacobian} Assume that  there exists a $\hat{\nu}-$Jacobian  $J$ on $X\times Y$  satisfying
\begin{equation}\label{eqpijacobi}
\iint f(x,y)J_0(x,y) \, \hat{\nu}^{y}(dx)d\pi(x,y) = \int f(x,y)\,d\pi(x,y),
\end{equation}
for any measurable function $f:X\times Y \to \mathbb{R}$.	 Then, $$H^{\hat{\nu}}(\pi) = -\int \log(J_0)\, d\pi.$$ 
\end{lemma}

\begin{proof}
	The reasoning is similar to the previous case.
	The set $A_0 =\{(x,y)\in X\times Y | J_0(x,y) > 0\}$ satisfies $\pi(A_0) =1$. For any $\hat{\nu}-$Jacobian $J:X\times Y \to (0,+\infty),$ satisfying $\inf_{x,y}J(x,y)>0$ we have
\[\int_{A_0} \log\left(\frac{J}{J_0}\right)\,d\pi \leq \log \int_{A_0}\frac{J}{J_0}\,d\pi = \log \int \frac{J}{J_0}\cdot I_{A_0} \cdot J_0\,\,\hat{\nu}^y(dx)d\pi(x,y) \]
\[= \log \int J\cdot I_{A_0} \,\,\hat{\nu}^y(dx)d\pi(x,y)\leq    \log \int J \,\,\hat{\nu}^y(dx)d\pi(x,y) = 0 .\]
	\end{proof}

We will say that a function $J_0$ satisfying \eqref{eqpijacobi} is {\bf a $\hat{\nu}-$Jacobian of} $\pi$.

\medskip

Denoting by $Q$ the $y-$marginal of $\pi$, the equation \eqref{eqpijacobi} can be rewritten as
\[ \iint f(x,y)J_0(x,y) \, \hat{\nu}^{y}(dx)dQ(y) = \int f(x,y)\,d\pi(x,y)\]
and so $J_0(x,y) \, \hat{\nu}^{y}(dx)dQ(y)$ is a disintegration of $\pi$ with respect to the horizontal lines of $X\times Y$. Supposing also that $\nu$ is an a priori probability on $X$, and $\hat{\nu}\equiv \nu$, we get that $\pi \ll \nu\times Q$, with $\frac{d\pi}{d\nu dq}=J_0$. Then, under the hypotheses of the above proposition and assuming that $\hat{\nu}\equiv \nu$, we get (see also \eqref{KLmutual})
\[H^{\nu}(\pi) = - D_{KL}(\pi\,|\,\nu\times Q).\]

\smallskip

\begin{lemma} Let $\hat{\nu}$ be a probability kernel and $\pi$ be a probability on $X\times Y$ with $y-$marginal $Q$. Suppose that $\pi \ll \hat{\nu}\, dQ$.  Then
	$$ H^{\hat{\nu}}(\pi)=- D_{KL}(\pi\,|\,\hat{\nu}\, dQ).$$
\end{lemma} 	

\begin{proof} We suppose that $\pi \ll \hat{\nu}\, dQ$ and we denote by $J$ its Radon-Nikodyn derivative. Then, for any measurable function $g:X\times Y \to \mathbb{R}$ we have
	\[\iint g(x,y) J(x,y) \, \hat{\nu}^y(dx)dQ(y) = \int g(x,y) \, d\pi(x,y).\]
	As $Q$ is the $y$-marginal of $\pi$, taking functions $g$ depending just of the second coordinate, we get
	\[\int g(y) [\int J(x,y) \, \hat{\nu}^y(dx)]dQ(y) = \int g(y) \, dQ(y).\]
 Then $\int J(x,y) \, \hat{\nu}^y(dx) = \frac{dQ}{dQ} =1$ for $Q-$a.e. $y$. Replacing $J$ by $1$ in a subset of $X\times Y$ having zero measure with respect to $\pi$, we get a measurable Jacobian $\tilde{J}$ satisfying, for any measurable  function $g:X\times Y \to \mathbb{R}$,
	\[\iint g(x,y) \tilde{J}(x,y) \, \hat{\nu}^y(dx)dQ(y) = \int g(x,y) \, d\pi(x,y).\]
	It follows from  Lemma \ref{pijacobian} that
	\[H^{\hat{\nu}}(\pi) = -\int \log(\tilde{J})\, d\pi=  - D_{KL}(\pi\,|\,\hat{\nu}\, dQ).\]
\end{proof}


\begin{lemma}\label{notabscont}  Let $\hat{\nu}$ be a probability kernel and $\pi$ be a probability on $X\times Y$, with $y-$marginal $Q$. If $\pi$ is not absolutely continuous with respect to $\hat{\nu}\, d Q$, then $H^{\hat{\nu}}(\pi) = -\infty$.
\end{lemma}
\begin{proof}
	If $\pi$ is not absolutely continuous 	
	with respect to $\hat{\nu}\, dQ$, then, there exists a measurable set $A\subset X\times Y$, such that, $\int \hat{\nu}^y(A)\,dQ(y)=0$ and $\pi(A)>0$.
	It follows that $\{y\,|\, \hat{\nu}^y(A) \neq 0\}$ is a measurable set on $Y$ satisfying $Q(\{y\,|\, \hat{\nu}^y(A) \neq 0\}) = 0$. The set $X\times \{y\,|\, \hat{\nu}^y(A) \neq 0\}$ is measurable in $X\times Y$ and, as the $y-$marginal of $\pi$ is $Q$, we get $\pi ( X\times \{y\,|\, \hat{\nu}^y(A) \neq 0\}) = Q(\{y\,|\, \hat{\nu}^y(A) \neq 0\}) = 0$. Let $B = A - (X\times  \{y\,|\, \hat{\nu}^y(A)) \neq 0\})$. The set $B$  is measurable and $\pi(B) =\pi(A)>0$, while $\hat{\nu}^y(B)=0\,\forall y \in Y$.
	
	For each $\beta>0$, let $c_\beta:X\times Y\to\mathbb{R}$ be the measurable function defined as
	\[ c_{\beta}(x,y) = \left\{\begin{array}{cl} 0 & \text{if}\, (x,y)\in B^C\\ \beta & \text{if}\, (x,y)\in B\end{array} \right. .\]
Then, for each fixed $y$ we have
	\[\int e^{c_\beta(x,y)}\,d\hat{\nu}^y(dx) = e^\beta\hat{\nu}^y(B) + e^{0}\hat{\nu}^y(B^C)   =1.\]
	This shows that $c_\beta$ is $\hat{\nu}-$normalized. Furthermore, $\int c_\beta(x,y)\, d\pi(x,y) =  \beta \pi(B)$.
	As $\beta$ is arbitrary we can take $\beta \to +\infty$ and then we get that  $H^{\hat{\nu}}(\pi) = -\infty$.
\end{proof}

The above results end the proof of Theorem \ref{theoKL}.

\bigskip

\begin{proposition}  Let $P$ be a probability on $X$ satisfying $P\ll \nu$. Consider any probability $Q$ on $Y$ and any probability $\pi$ on $X\times Y$, with $x-$marginal $P$. Then, we have:\newline
	1. $H^{\nu}(\pi) \leq S^{\nu}(P)$\newline
	2. $S^{\nu}(P) = H^{\nu}(P\times Q).$
\end{proposition}	

\begin{proof}
	The proof of item 1. is a direct consequence of the definitions of $S^{\nu}$ and $H^{\nu}$ because we can consider any measurable function $c:X\to\mathbb{R}$, as a measurable function defined on $X\times Y$ which depends just on the first coordinate.
	
	In order to prove item 2. we consider the function $J(x,y) = \frac{dP}{d\nu}(x)$. This function is a $\nu-$Jacobian of $P\times Q$, then, applying   propositions \ref{kullback} and \ref{pijacobian}, we conclude the proof.
\end{proof}

The proof of the next result follows the same reasoning which was used in \cite{LMMS} and \cite{LM}.


\begin{proposition} The entropy $H^{\hat{\nu}}(\cdot)$ has the following properties:\newline
	1. $H^{\hat{\nu}} $ is concave\newline
	2. $H^{\hat{\nu}}$ is upper semi continuous. More precisely, if $\int f d\pi_n \to \int f\,d\pi$, for any measurable and bounded function $f$ on $X\times Y$,  then, $\displaystyle{\limsup_n H^{\hat{\nu}}(\pi_n) \leq H^{\hat{\nu}}(\pi).}$
\end{proposition}

\begin{definition}\label{defIG} We define the \textbf{information gain} of a probability $\pi$ on $X\times Y$, with respect to the a priori probability kernel $\hat{\nu}$, as
	\[IG(\pi,\hat{\nu}) = -H^{\hat{\nu}}(\pi).\]
\end{definition}

If $\pi$ has marginals $P$ and $Q$ and $\pi \ll P\times Q$, then choosing $\hat{\nu}\equiv P$ we get
\[IG(\pi, P) = -H^{P}(\pi) = D_{KL}(\pi\,|\,P\times Q),\]
which corresponds to the mutual information given by \eqref{KLmutual}.  As $c=0$ is $\hat{\nu}-$normalized we get $IG(\pi,\hat{\nu})\geq 0$. Furthermore,  $IG (\pi,\hat{\nu})=0$, if $d\pi = \hat{\nu}^{y}(dx)dQ(y)$, for some $Q$.

The information gain $IG(\pi,P)$ above defined can be computed from $S^{\nu}(P)$ and $H^{\nu}(\pi)$, and it ``does not depend'' on the choice of the {\it a priori} probability $\nu$ on $X$ as the following result shows.

\begin{proposition} \label{ma1} Let $P$ be a probability on $X$, $\pi$ be a probability on $X\times Y$, with $x-$marginal $P$ and let $\nu$  be an a priori probability on $X$.  Assume that $S^{\nu}(P)$ and $H^{\nu}(\pi)$ are finite. Then,
	\[ IG(\pi,P)=S^{\nu}(P) - H^{\nu}(\pi)  .\]
\end{proposition}

\begin{proof}
By hypothesis there exists  $\phi:X\to[0,+\infty)$, a $\nu-$Jacobian of $P$ and $J:X\times Y \to [0,+\infty)$, which is a $\nu-$Jacobian of $\pi$. Denoting by $Q$ the $y-$marginal of $\pi$, we have $d\pi(x,y) = J(x,y) d\nu(x)dQ(y)$ and $dP(x) = \phi(x)d\nu(x)$. The set $A =\{x\in X\,|\,\phi(x)>0\}$ satisfies  $P(A)=1$ and, as $\pi$ has $x-$marginal $P$, we finally get $\pi(A\times Y) = 1$. So we can write
$ d\pi(x,y)  = \frac{J(x,y)}{\phi(x)}dP(x)dQ(y)$ 	and therefore
\[ 	S^{\nu}(P) - H^{\nu}(\pi) = - \int \log(\phi)dP + \int \log(J)d\pi  =-H^{P}(\pi) =  IG(\pi,P) .
\]

\end{proof}

\begin{proposition} Let  $\hat{\nu}$ be an a priori probability kernel and $\pi$ be a probability on $X\times Y$. Given a bounded and $\hat{\nu}-$normalized function $\phi_0:X\times Y\to\mathbb{R}$, consider the a priori probability kernel $\hat{\mu}^y(dx) = e^{\phi_0(x,y)}\hat{\nu}^y(dx)$. Then,
	\[IG(\pi, \hat{\mu}) = -\int \phi_0\,d\pi + IG(\pi, \hat{\nu}).\]
\end{proposition}

\begin{proof}
	\[IG(\pi, \hat{\mu}) = -H^{\hat{\mu}}(\pi) = \sup \{\int c\,d\pi\,|\, \int e^{c(x,y)}\,\hat{\mu}^y(dx) = 1 \,\forall y\}\]
	\[= \sup \{\int c\,d\pi\,|\, \int e^{c+\phi_0}\,\hat{\nu}^y(dx) = 1 \,\forall y\}\]
	\[=\sup \{\int c-\phi_0\,d\pi\,|\, \int e^{c}\,d\hat{\nu}^y(dx) = 1 \,\forall y\}\]
	\[=-\int \phi_0\,d\pi - H^{\hat{\nu}}(\pi).\]
\end{proof}

\begin{corollary} \label{IGequivalents} Let  ${\nu}$ be an a priori probability on $X$ and $\pi$ be a probability on $X\times Y$. Given a bounded and ${\nu}-$normalized function $\phi_0:X\times Y\to\mathbb{R}$, consider the a priori probability kernel $\hat{\mu}^y(dx) = e^{\phi_0(x,y)}d\nu(x)$. Then,
	\[IG(\pi, \hat{\mu}) = -\int \phi_0\,d\pi - H^{\nu}(\pi).\]
\end{corollary}

 This last result shows that the above definition of information gain, using probability kernels, is a natural extension of \eqref{IG2} (see also Remark \ref{remarkidea}). Given a probability $Q_0$ on $Y$ we can associate a probability $\pi_0$ on $X\times Y$ given by $d\pi_0 = e^{\phi_0(x,y)}d\nu(x) dQ_0(y)$. A natural generalization of  \eqref{IG2}, in principle, could be given by
\begin{equation}\label{IGnotwell} IG(\pi,\pi_0) =   -\int \phi_0\,d\pi - H^{\nu}(\pi),\end{equation}
but we remark that in some cases it is not even well defined. Indeed, if $Y=\{0,1\}$  and $\pi_0 = \nu\times \delta_0$, the functions $\phi_0$ and $\psi_0$ given by $\phi_0(x,y)= 0$ and $\psi_0(x,y)=y\cdot f(x)$, where $f\neq 0$ is a $\nu$-normalized function,  provide two different disintegrations of $\pi_0$, which are
\[d\pi_0(x,y) = e^{\phi_0(x,y)} d\nu(x)d\delta_0(y) \,\,\,\,\text{ and}\,\,\,\, d\pi_0(x,y) = e^{\psi_0(x,y)} d\nu(x)d\delta_0(y).\]
If $\pi=\nu\times \delta_{1}$, then $H^{\nu}(\pi) =H^{\nu}(\nu\times\delta_1)= S^{\nu}(\nu) = 0$, and so
\[-\int \psi_0\,d\pi - H^{\nu}(\pi) = -\int \psi_0\,d\pi = \int f(x)\,d\nu(x)\] while
\[-\int \phi_0\,d\pi - H^{\nu}(\pi) = 0.\]

The problem concerning the extension of \eqref{IG2} for measurable spaces can be  solved by using probability kernels. We observe that the right-hand side of \eqref{IGnotwell} contains just  $\pi,\phi_0$ and $\nu$. Therefore, we realize that we are not using $Q_0$ and the associated $\pi_0$, but only the probability kernel $\hat{\mu}=e^{\phi_0}d\nu$, which is part of a disintegration of $\pi_0$. In this sense, it is natural to define information gain by using probability kernels and to consider $IG(\pi, \hat{\mu})$ instead of trying to define $IG(\pi,\pi_0)$.

\section{Entropy and information gain for compact metric spaces }\label{seccompact}

In this section, we consider compact metric spaces   $X$ and $Y$ equipped  with their respective Borel $\sigma$-algebras. We will prove that, in this case, $H^{\nu}$ coincides with the entropy considered in \cite{MO}, which is defined from a supremum taken over Lipschitz functions instead of measurable functions. This shows that the concept of entropy, as defined in \cite{LM} for metric spaces, which is also considered here for measurable spaces, extends the concept of entropy as described in \cite{LMMS} and \cite{MO}. This also shows that such entropies are given by simple expressions concerning $D_{KL}$.   In the  second part of this section, we also propose to rewrite the information gain by using ``special'' probabilities $\pi_0$ instead of  probability kernels $\hat{\nu}$ (see the end of the previous section). This will be coherent with the reasoning  of future sections and also extends \eqref{eq3}, \eqref{eq3new} and \eqref{IG2} from finite sets (finite alphabet) to compact spaces.

\begin{theorem}\label{entropylip} Suppose that $X$ and $Y$ are compact metric spaces and consider the Borel sigma-algebras in $X$ and $Y$. Then,
	\[H^{\nu}(\pi)= -\sup \{ \int f(x,y)\,d\pi(x,y) | \int e^{f(x,y)}\,d\nu(x) = 1,\,\forall y,\, \, \text{with}\, f\, \, \text{Lipschitz}\}.\]
\end{theorem}

\begin{proof} By definition
	\[H^{\nu}(\pi)= -\sup \{ \int f(x,y)\,d\pi(x,y) | \int e^{f(x,y)}d\nu(x) = 1\,\forall y,\, f\,\in \mathcal{F}(\pi)\}.\]
	We denote
	\[h^{\nu}(\pi):= -\sup \{ \int f(x,y)\,d\pi(x,y) \,|\, \int e^{f(x,y)} d\nu(x) = 1,\,\forall y,\, \text{with}\,f\, \text{Lipschitz}\}.\]

It will be  necessary to prove that $H^{\nu}=h^{\nu}$.
	
	\bigskip
	
If $\psi:X\times Y \to\mathbb{R}$ is Lipschitz and $\nu-$normalized, $ Q$ is any probability on $Y$ and $\pi_{\psi}$ is the probability on $X\times Y$, given by
 \[ \int f(x,y)\, d\pi_\psi(x,y) := \iint f(x,y)e^{\psi(x,y)}\,d{\nu}(x)dQ(y),\,\text{for}\,\,f \,\,\text{measurable},\]
 then, $e^{\psi}$ is a $\nu$-Jacobian of $\pi_\psi$. It follows in this case, that
 \[H^{\nu}(\pi_\psi)=   -\int \psi(x,y)\,d\pi_\psi(x,y) = h^{\nu}(\pi_\psi). \]

Suppose by contradiction there exists a probability $\eta$ on $X\times Y$, such that, $-H^{\nu}(\eta)>-h^{\nu}(\eta)$. Consequently, $-h^{\nu}(\eta)\neq +\infty$.

First, we claim that there exists a Lipschitz function $\varphi:X\times Y\to\mathbb{R}$, such that, for any probability $\pi$ on $X\times Y$,
\[ \int \varphi \, d\eta +h^{\nu}(\eta) > \int \varphi \, d\pi +H^{\nu}(\pi).\]

The proof of this claim follows the same reasoning  of the proof of Theorem 3 in \cite{ACR} (see also  \cite{ET} chap. 1).
We consider the weak* topology on the space of finite signed-measures and we extend $H^{\nu}$ and $h^{\nu}$ as the value $-\infty$, if $\pi$ is not a probability.

As $-H^{\nu}$ is convex, non negative and lower semi-continuous, its epigraph $\operatorname{epi}(-H^{\nu})=\{(\pi,t) \,|\, -H^{\nu}(\pi) \leq t\}$ is convex and closed. As $(\eta,-h^{\nu}(\eta))\notin \operatorname{epi}(-H^{\nu})$, it follows from Hahn-Banach theorem that there is  $c\in\mathbb{R}$ and a linear functional
\[(\pi,t) \to \int g\,d\pi + at,\]
where $g$ is a fixed continuous function and $a\in \mathbb{R}$ is fixed, such that, for any $(\pi,t) \in \operatorname{epi}(-H^{\nu})$ we have
\[\int g \,d\eta - ah^{\nu}(\eta) <c<\int g \,d\pi + at.\]
Observe that necessarily $a>0$. We denote $\varphi = -\frac{g}{a}$. If a probability $\pi$ satisfies $-H^{\nu}(\pi)<+\infty$, then $(\pi,-H^{\nu}(\pi)) \in \operatorname{epi}(-H^{\nu})$ and finally  we get
\[\int \varphi \,d\eta +h^{\nu}(\eta) >-\frac{c}{a}>\int \varphi \,d\pi  +H^{\nu}(\pi).\]
Even in the case  $-H^{\nu}(\pi)=+\infty$ these inequalities remain valid.

Finally, as the set of Lipschitz functions is dense in the set of continuous functions (in the uniform convergence) we can assume that for a Lipschitz function $\varphi$ we have
 \[ \int \varphi \, d\eta +h^{\nu}(\eta) > \int \varphi \, d\pi +H^{\nu}(\pi),\]for any probability $\pi$. This finishes the proof of the claim.

\bigskip

Let $Q$ be the $y$-marginal of $\eta$ and $\tilde{\varphi}(y) = \log(\int e^{\varphi(x,y)}d\nu(x))$. The function $\psi(x,y) = \varphi(x,y) - \tilde{\varphi}(y)$ is Lipschitz, $\nu-$normalized and for any probability $\pi$ on $X\times Y$, with $y-$marginal $Q$, we have
\[ \int \psi \, d\eta +h^{\nu}(\eta)> \int \psi \, d\pi +H^{\nu}(\pi).\]

Let $\pi = \pi_{\psi}$ be defined by
 \[ \int f(x,y)\, d\pi_\psi(x,y) := \iint f(x,y) e^{\psi(x,y)}\,d\nu(x)dQ(y).\]
Then, as $\pi_\psi$ has $y-$ marginal $Q$ and $H^{\nu}(\pi_\psi)=-\int \psi \, d\pi_\psi$, we get that
\[  \int \psi \, d\eta + h^{\nu}(\eta) >  \int \psi \, d\pi_\psi +H^{\nu}(\pi_\psi) = 0 . \]
This is a contradiction because, by definition of $h^{\nu}$, $ \int \psi \, d\eta + h^{\nu}(\eta)\leq 0$.
\end{proof}

The next results will be necessary later; in the direction of getting a different point of view for the concept of  information gain.

\begin{proposition} \label{equivalent}
Let $X$ and $Y$ be compact metric spaces and suppose that  there exists a continuous/Lipschitz function $\phi:X\times Y \to \mathbb{R}$, such that, $J=e^{\phi}$ is a $\nu$-Jacobian of $\pi$. Let $P$ and $Q$ be the marginals of $\pi$ on $X$ and $Y$. Then, $P$ is equivalent to $\nu$ (each one is absolutely continuous with respect to each other) and
	$$\frac{d P}{d\nu}(x) =\int e^{{\phi}(x,y)}dQ(y),$$ which is also continuous/Lipschitz. There exist constants $c_2>c_1>0$, such that, $c_1<\frac{d P}{d\nu}<c_2$,\, $\forall x\in X$. Finally, defining $\psi := \phi- \log(\frac{d P}{d\nu})$, we have that $e^\psi$ is a $P-$Jacobian of $\pi$.
\end{proposition}

\begin{proof}
	If a measurable and bounded function $g$ depends only of the first coordinate, then, using Fubini's theorem, we get
\begin{align*} \int g(x) d P(x) &= \int g(x) d\pi(x,y) = \iint e^{{\phi}(x,y)}g(x)\,d\nu(x)dQ(y)\\& = \int[\int  e^{{\phi}(x,y)}dQ(y)] g(x)\,d\nu(x).\end{align*}
	It follows that  $\frac{d P}{d\nu}(x) =\int e^{{\phi(x,y)}}dQ(y)$ is a density.
		Clearly the function $\frac{d P}{d\nu}$ is continuous/Lipschitz and there are constants $c_2>c_1>0$, such that, $c_1<\frac{d P}{d\nu}(x)<c_2$,\, $\forall x\in X$. This shows that $P$ and $\nu$ are equivalent and that $\log(\frac{d P}{d\nu})$ is continuous/Lipschitz.
	
	Let $\psi(x,y):=\phi(x,y)-\log(\frac{d P}{d\nu})(x)$. Then, for any measurable and bounded function $g:X\times Y \to\mathbb{R}$  we have
	\[\int e^{\psi(x,y)}g(x,y)\,d P(x)= \int e^{\phi(x,y)} g(x,y) {d\nu}(x) .\]
By integrating both sides with respect to $Q$ and using the fact that $e^{\phi}$ is a $\nu-$Jacobian of $\pi$, we get
\[\int[\int e^{\psi(x,y)}g(x,y))\,d P(x)]dQ(y) = \int g(x,y)d\pi(x,y).\]
		This shows that $\psi$ is a $P-$Jacobian of $\pi$.
\end{proof}

\begin{proposition}\label{uniqueJlip}  Suppose that $X$ and $Y$ are compact metric spaces. Let $\nu$ be a probability on $X$ with $\operatorname{supp}(\nu) = X$ and $\phi:X\times Y\to\mathbb{R}$ be a continuous function. Let $\pi_0$ be a probability on $X\times Y$ with $\nu$-Jacobian $J=e^{\phi}$.  If $\pi_0$ is positive on open sets of $X\times Y$, then $\phi$ is the unique continuous function, such that, $e^{\phi}$ is a $\nu-$Jacobian of $\pi_0$.
\end{proposition} 	

\begin{proof}
Let $Q$ be the $y$-marginal of $\pi_0$ and $\phi_2$  be a continuous function, such that, $e^{\phi_2}$ is a $\nu-$Jacobian of $\pi_0$. Then,
\[d\pi_0(x,y) = e^{\phi(x,y)}d\nu(x)dQ(y) \,\,\,\text{and}\,\,\, d\pi_0(x,y) = e^{\phi_2(x,y)}d\nu(x)dQ(y).\]
This shows that the positive functions $e^{\phi_2}$ and $e^{\phi}$ satisfy $e^{\phi_2(x,y)}=e^{\phi(x,y)}$, $\pi_0$-a.e. $(x,y) \in X\times Y$. 
As $\pi_0$ is positive on open sets and $\phi,\phi_2$ are also continuous, we get $\phi_2(x,y)=\phi(x,y)$, for all $(x,y) \in X\times Y$.
\end{proof}

\begin{definition} [Information Gain for compact spaces]  \label{defIGall2}
	Let $X$ and $Y$ be compact metric spaces and $\pi_0$ be a probability on $X\times Y$ which is positive on open sets and has $x-$marginal $P_0$. Suppose there exists a continuous function $\psi_0:X\times Y\to\mathbb{R}$, such that, $e^{\psi_0}$ is a $P_0-$Jacobian of $\pi_0$. For any probability $\pi$ on $X\times Y$ we define the information gain of $\pi$ with respect to $\pi_0$ as
	\begin{equation}\label{IGall2}
	IG(\pi,\pi_0) = -\int \psi_0\, d\pi - H^{ P_0}(\pi).
	\end{equation}
\end{definition}

Suppose that $\pi_0$ has $y-$marginal $Q_0$ and  $d\pi_0(x,y) = e^{\psi_0(x,y)}dP_0(x)dQ_0(y)$. Observe that in the expression $-\int \psi_0\, d\pi - H^{ P_0}(\pi)$ of \eqref{IGall2} appears $\psi_0$ and $P_0$ but not $Q_0$. By considering the probability kernel $\hat{\nu}^y(dx) =e^{\psi_0(x,y)}dP_0(x)$ the next result shows that this definition is coherent with Definition \ref{defIG}.

\begin{proposition}
	Under the assumptions of Definition \ref{defIGall2}, we define $\hat{\nu}^y(dx)=e^{\psi_0(x,y)}dP_0(x)$. Then,
	\[ IG(\pi,\pi_0) =  IG(\pi,\hat{\nu}) . \]
\end{proposition}

\begin{proof} It is a consequence of corollary \ref{IGequivalents}.
\end{proof}

The next proposition shows that the above interpretation of information gain does not depend, in a sense to be explained, on the choice of the {\it a priori} probability $\nu$.

\begin{proposition}\label{IGanynu} Let $X$ and $Y$ be compact metric spaces and $\pi_0$ be a probability on $X\times Y$ which is positive on open sets. Let $\nu$ be an a priori probability on $X$. Suppose there exist a continuous function $\phi$, such that, $e^{\phi}$ is a $\nu$-Jacobian of $\pi_0$. Let $\pi$ be any probability on $X\times Y$. Then,
	\[IG(\pi,\pi_0)=-\int \phi \, d\pi - 	H^{\nu}(\pi).\]
\end{proposition}

\begin{proof}
	Let $P$ and $Q$ be the marginals of $\pi_0$.
		From lemma \ref{equivalent}, $P$ is equivalent to $\nu$, with $\frac{d P}{d\nu}(x) =\int e^{{\phi}(x,y)}dQ(y)$. Furthermore,  the continuous function $\psi := \phi- \log(\frac{d P}{d\nu})$ is such that $e^\psi$ is the $P-$Jacobian of $\pi_0$. Clearly, $e^{\psi-\phi}=\frac{d \nu}{dP}$, and so a Lipschitz function $c$ is $\nu$-normalized, if and only if, $c +\psi -\phi$ is $P$-normalized.
	
	It follows that for any probability $\pi$ we get
	\[ H^{\nu}(\pi) = H^{P}(\pi) + \int \psi-\phi\,d\pi.\]
	Therefore,
	\[ IG(\pi,\pi_0) = -\int \psi\, d\pi - H^{ P}(\pi) = -\int \phi\, d\pi - H^{\nu}(\pi).\]	
\end{proof}

From the above, it is legitimate  to say that \eqref{IGall2} generalizes \eqref{IG2}. The next example shows that in a certain sense to be explained,  \eqref{IGall2}  also generalizes \eqref{eq3} and \eqref{eq3new}.

\begin{example} Suppose $X=\{1,2,...,d\}$, $Y=\{1,2,...,d\}^{\{2,3,4,5,...\}}$ and identify $\Omega$ with $X\times Y$ by the homeomorphism
\begin{equation}\label{omegatoxy}  \Omega \ni |x_1,x_2,x_3,x_4,...) \to (x_1,\,|x_2,x_3,x_4,...)) \in X\times Y.  \end{equation}
	When considering the a priori probability $\nu$ as the counting measure on $X$, we get that a $\nu-$Jacobian of an invariant probability $\pi$ is given by
	\[J^\pi(x_1,x_2,x_3,...) = \lim_{n\to\infty} \frac{\pi(x_1,x_2,....,x_n)}{\pi(x_2,x_3,...x_n)},\]
for $\pi$ a.e. $x\in \Omega$ (it can be extended for any point of $\Omega$ by taking $J^{\pi} =\frac{1}{d}$ in a set of zero measure). Furthermore, the Kolmogorov-Sinai entropy of  $\pi$ coincides with $H^{\nu}(\pi)$.
	A measurable function $\phi:\Omega\to\mathbb{R}$ is normalized if $\sum_{a}e^{\phi(|a,x_2,x_3,...))} =1$, for all $|x_2,x_3,...)$. If $\phi$ is Lipschitz with corresponding equilibrium probability $\pi_{\phi}$, then the unique continuous Jacobian of $\pi_\phi$ is $e^{\phi}$ and the pressure of $\phi$ is zero. It follows from \eqref{eq3} and  Proposition \ref{IGanynu} that for any invariant probability $\pi$ we have
	\[h(\pi,\pi_\phi) = IG(\pi,\pi_\phi).\]
\end{example}

\section{Specific information gain in the TFCA model }\label{xy}

In this section, we introduce a dictionary connecting the definitions and results of previous sections with analogous ones for  the TFCA model. This will allow us to introduce the specific information gain in this setting; a necessary step for introducing  the concept of entropy production in the next section.  First, we will remember some of the main definitions and results for the TFCA model  described  in \cite{LMMS} and which will be  extensively used here.

Let $(M,d)$ be a compact metric space and denote by $\Omega=\Omega^{+}$ the space $M^{\mathbb{N}}$. Elements in $\Omega$ will be written in the form $x=|x_1,x_2,x_3,...),\, x_i \in M$. The space $\Omega$ is compact using the metric
$d(x,y) = \sum_{i=1}^{\infty}\frac{d(x_i,y_i)}{2^{i}}.$
We also consider the Borel sigma-algebra in $\Omega$.

The relation of the  setting of this section with the previous ones can be clarified  by considering $X=M$, $Y=M^{\{2,3,4,5,...\}}$ and identifying $\Omega$ with $X\times Y$, using the homeomorphism given in \eqref{omegatoxy}.
 Observe that $Y$ can be also identified with $X\times Y$ using the homeomorphism
$|x_2,x_3,x_4,...)\to (x_2,|x_3,x_4,x_5,...))$.
From this identification the shift map ${\sigma}:\Omega\to \Omega$ given by $ {\sigma}(|x_1,x_2,x_3,...)) = |x_2,x_3,...)$ can be also interpreted as the projection on $Y$.  We say that a probability $\mu$ on $\Omega=M^{\mathbb{N}}$ is invariant for the shift map $\sigma$ (or, shift-invariant), if for any continuous function $f:\Omega\to\Omega$, we have $\int f\, d\mu = \int f\circ\sigma\, d\mu$.

Assume we  fixed an {\it a priori} probability $\nu$ on $M$ satisfying $\operatorname{supp}(\nu) = M$. For each Lipschitz function $A:\Omega\to\mathbb{R}$ we consider the linear operator $\mathcal{L}_{A,\nu}:C(\Omega) \to C(\Omega)$ defined by
\[\mathcal{L}_{A,\nu}(f)(x) = \int e^{A(|a,x_1,x_2,x_3,...))}f(|a,x_1,x_2,x_3,...))\,d\nu(a),\,\,x=|x_1,x_2,x_3,...).\]

We call $\mathcal{L}_{A,\nu}$ the Ruelle operator (or, transfer operator) associated to the Lipschitz potential $A$ and the {\it a priori} probability $\nu$ (we refer the reader to \cite{LMMS} for general properties of this operator).

For this operator
there exists a unique (simple) positive eigenvalue $\lambda_A$ associated to a positive eigenfunction $h=h_A$. 
  If a continuous function $h>0$ satisfies  $\mathcal{L}_{A,\nu}(h) =\lambda_A \cdot h$, then $h$ is Lipschitz. If $h_1$ and $h_2$ are eigenfunctions associated to $\lambda_A$, then $h_2 = c\cdot h_1$ for some constant $c$. There exists a unique probability measure $\rho_A$ on  $\Omega$ satisfying $\mathcal{L}_{A,\nu}^*(\rho_A) = \lambda_A \cdot \rho_A$, which  means that
\[\int \mathcal{L}_{A,\nu}(f)\,d\rho_A = \lambda_A \int f\,d\rho_A,\]
for any continuous function $f:\Omega\to\mathbb{R}$. For a matter of convenience we fix the eigenfunction $h_A$ which satisfies $\int h_A\, d\rho_A = 1$.

\medskip

We point out that in the case the space of symbols $M$ is not countable we really need to introduce an {\it a priori} probability in order to get a transfer operator.
\bigskip

A Lipschitz function $A$ is called $\nu$-normalized if $\mathcal{L}_{\bar{A},\nu}(1) = 1$, that is,
\[ \int e^{A(|a,x_1,x_2,x_3,...))}\,d\nu(a)=1,\,\,\forall\,x=|x_1,x_2,x_3,...).\]
The function
\begin{equation} \label{coh} \bar{A} = A+\log(h_A) - \log(h_A\circ \sigma) - \log(\lambda_A)\end{equation}
is $\nu$-normalized. The associated eigenprobability $\rho_{\bar{A}}$ is shift-invariant and it will be denoted also by $\mu_A$. It also satisfies $d\mu_A = h_Ad\rho_A$ and $\mathcal{L}_{\bar{A},\nu}^*(\mu_A) = \mu_A$.

 The {\bf relative entropy} of an invariant probability $\mu$ on $\Omega$ with respect to the {\it a priori} probability $\nu$ on $M$ is defined in \cite{LMMS} as
\[ h^\nu(\mu) = -\sup_{B\,is\,\nu- normalized} \int B\, d\mu.\] Considering the above identification of $\Omega$ and $X\times Y$ and applying Theorem \ref{entropylip} we see that this definition is consistent with the previous definition of $H^{\nu}$.
Let us formally enunciate this result (using also Theorem \ref{theoKL}).

\begin{theorem} Denoting by $\pi$ the probability on $X\times Y$, which corresponds to the shift-invariant probability $\mu$ on $\Omega$, we get that the relative entropy  $h^\nu(\mu)$, as defined in \cite{LMMS}, coincides with $H^{\nu}(\pi)$. Furthermore, if $Q$ is the $y-$marginal of $\pi$ (which is identified with $\pi$ because $\mu$ is shift invariant), then we have
\[h^{\nu}(\mu) = H^{\nu}(\pi) = -D_{KL}(\pi\,|\,\nu\times Q).\]	
\end{theorem}

In \cite{ACR} it is proved that $h^{\nu}(\mu)$ coincides with the so called specific entropy of Statistical Mechanics (see \cite{Geo}).

\bigskip 	

For any Lipschitz function $A:\Omega\to\mathbb{R}$ we have $h^{\nu}(\mu_A) = -\int \bar{A}\,d\mu_A.$
Furthermore,
\[P_\nu(A):=\sup_{\mu\,\, shift-invariant} \int A\, d\mu +h^{\nu}(\mu ) = \int A\,d\mu_A + h^{\nu}(\mu_A) = \log(\lambda_A).\]
The number $P_{\nu}(A)$ is called  the $\nu-$pressure of $A$.
A probability $\mu$ attaining the supremum value $P_\nu(A)$ is called an {\bf  equilibrium probability} for  $A$. 
In \cite{ACR} it is proved that  $\mu_A$ is the unique equilibrium probability for $A$.

If $M$ is a finite set with $d$ elements and the {\it  a priori} probability $\nu$ is set to be the counting measure (which is not a probability), then the relative entropy $h^\nu(\mu)$ above defined coincides with the Kolmogorov-Sinai entropy $h(\mu)$ (see prop.7 in \cite{LMMS} and lemma 7 in \cite{LMMS2}). If we set  the {\it a priori} probability $\nu$ as the normalized counting measure on $M$ (which is a probability), then $h^\nu(\mu)= h(\mu) - \log  d\leq 0$.

\bigskip

We remark that all the above concepts  depend on the choice of the a priori probability on $M$. A universe of different choices is possible. If $\nu_1$ and $\nu_2$ are two different a priori probabilities which are not equivalent, then a probability $\mu$ on $\Omega$ could be a $\nu_1$-equilibrium probability, and, at the same time,  it not to be a $\nu_2$-equilibrium probability (see  Proposition \ref{Pmuequivalent}).

We will call an invariant probability $\mu$ on $\Omega$ of  \textbf{equilibrium probability} if there exists at least one {\it a priori} probability $\nu$, satisfying $\operatorname{supp}(\nu)=M$, and also  a Lipschitz $\nu$-normalized function $A$, such that, $\mathcal{L}_{(A,\nu)}^{*}(\mu) = \mu$. In this case the probability $\mu$ is the $\nu$-equilibrium probability for the  normalized potential $A$. If  $A$ is a Lipschitz function  which is not normalized we can apply the  construction given by (\ref{coh}) and we get an associated normalized potential $\bar{A}$ (which is also Lipschitz). The probability $\mu$ is the $\nu$-equilibrium probability for both functions  $A$ and $\bar{A}$.

Let $A:\Omega \to \mathbb{R}$ be a $\nu-$normalized Lipschitz function. We will call $e^{A}$ of a $\nu-$Jacobian of the shift-invariant probability $\mu$ if for any continuous function $g:\Omega \to \mathbb{R}$ we have
\[ \iint e^{A(|a,x_2,x_3,...))}g(|a,x_2,x_3,...))d\nu(a)d\mu(|x_2,x_3,...)) = \int g(x)d\mu(x).\]
The following statements are equivalent for a $\nu-$normalized Lipschitz function $A$ and a shift-invariant probability $\mu$ on $\Omega$:\newline
i. $e^{A}$ is a $\nu-$Jacobian of  $\mu$\newline
ii. $\mu$ is the $\nu$-equilibrium probability  $A$\newline
iii. $\mathcal{L}_{(A,\nu)}^{*}(\mu) = \mu$.

Given an equilibrium probability $\mu$ we denote by $ P_\mu$ the projection of $\mu$ on the first coordinate. This means that for any continuous function $g:\Omega\to\mathbb{R}$, which depends only of the first coordinate, we have
\[ \int_M g(a) d P_\mu(a) := \int_\Omega g(x_1) d\mu(|x_1,x_2,x_3,...)).\]
 The  next result is a corollary of Proposition \ref{equivalent}.

\begin{proposition}\label{Pmuequivalent}
Suppose that $\mu$ is an equilibrium probability and let $ P_\mu$ be the projection of $\mu$ on the first coordinate.  If $\mu$ is the $\nu$-equilibrium probability for the Lipschitz normalized  potential $A$, then $\nu$ is equivalent to $ P_\mu$ and
$$\frac{d P_\mu}{d\nu}(a) =\int e^{{A}(|a,x_1,x_2,x_3...))}d\mu(|x_1,x_2,x_3,...)),$$ which is also Lipschitz. 
\end{proposition}

It is known that any  equilibrium probability $\mu$ is positive on open sets of $\Omega=M^\mathbb{N}$  (see Prop. 3.1.8. in \cite{Melo} - see also \cite{CMRS}). Therefore, the next result is a corollary of Proposition \ref{uniqueJlip}.

\begin{proposition}\label{uniquejacob} Let $\mu$ be an equilibrium probability and let $\nu$ be an a priori probability. Suppose $A$ is a Lipschitz $\nu-$normalized function, such that, $\mu$ is the $\nu-$equilibrium for $A$, then $A$ is the unique Lipschitz $\nu-$Jacobian of $\mu$.
\end{proposition}

	\bigskip
	
	The next definition is inspired by (\ref{eq3}) and \eqref{IGall2}.

\begin{definition} Let $\eta$ be a shift-invariant probability and $\mu$ be an equilibrium probability.   Then, we define the \textbf{specific information gain} of $\eta$, with respect to $\mu$, by
\begin{equation}\label{infgain}
h(\eta,\mu) = \left[\int B\,d\mu +h^{ P_\mu}(\mu)\right]-\left[\int B\, d\eta - h^{ P_\mu}(\eta)\right],
\end{equation}
where $B$ is any Lipschitz function, such that, $\mu$ is the $ P_\mu$-equilibrium probability of $B$.
\end{definition}

Observe that  by the  variational principle the information gain is $\geq 0$.

 There exists a unique $ P_\mu$-normalized function $\bar{B}$, such that, $\mu$ is the $ P_\mu$-equilibrium for $\bar{B}$. If $B$ is not normalized, then there exists a positive function $h_B$ and a positive number $\lambda_B$, such that,
 $$\bar{B} = B + \log(h_B) - \log(h_B)\circ\sigma - \log(\lambda_B).$$
 It follows that
\[ \int B\,d\mu +h^{ P_\mu}(\mu)-\int B\, d\eta - h^{ P_\mu}(\eta)  =  \int \bar{B}\,d\mu +h^{ P_\mu}(\mu)-\int \bar{B}\, d\eta - h^{ P_\mu}(\eta).\]
This shows that $h(\eta,\mu)$ is well defined (it does not change if either $B$ is $ P_\mu$-normalized, or not).
We remark that if $B$ is (the unique possible) normalized potential, then $\int B\,d\mu +h^{ P_\mu}(\mu)=0$ and therefore we get the following result which is a particular version of \eqref{IGall2}.

\begin{proposition} If $\mu$ is an equilibrium probability and  $e^{B}$ is the Lipschitz $P_{\mu}-$Jacobian of $\mu$, then
$$h(\eta,\mu) = -\int B\, d\eta - h^{ P_\mu}(\eta).$$
\end{proposition}

The above definition considers, for an equilibrium probability $\mu$, the a priori probability $ P_\mu$. In this way, the {  previous definition of specific information gain does not allow a choice of} $\nu$. The next  result, which is a corollary of Proposition \ref{IGanynu},  shows that if we exchange $P_\mu$ by another {\it a priori} probability $\nu$, then, it is true a similar formula for $h(\eta,\mu)$. This means, that the information gain does not depend on the particular choice of $\nu$, as long as $\mu$ is a $\nu$-equilibrium probability.

\begin{proposition}\label{anynu} Consider any a priori probability $\nu$ and any Lipschitz function $A$, such that, $\mu$ is the $\nu$-equilibrium probability for $A$. Let $\eta$ be any invariant probability. Then $h(\eta,\mu)$ as defined in (\ref{infgain})  satisfies
\begin{equation}
h(\eta,\mu) =  \left[\int A\,d\mu +h^{\nu}(\mu)\right]-\left[\int A\, d\eta - h^{\nu}(\eta)\right].
\end{equation}
 \end{proposition}

\begin{proof} First note that if we replace $A$ by its normalization $\bar{A}$, then the value on the right hand side of the above  expression  does not change. Then, we can suppose that $A$ is $\nu-$normalized. Therefore, it is just necessary  to prove that
$h(\eta,\mu) =  -\left[\int A\, d\eta - h^{\nu}(\eta)\right]$.	
But, this follows from Proposition \ref{IGanynu}.
\end{proof}


\begin{example} Consider any a priori probability $\nu$ and the Lipschitz function $A=0$. We observe that $\bar{\nu}=\nu\times\nu\times\nu\times...$ is the $\nu$-equilibrium probability for $A=0$. Given any invariant probability $\eta$,  then
	\begin{equation}
	h(\eta,\bar{\nu}) =  - h^{\nu}(\eta).
	\end{equation}
Therefore, the specific information gain generalizes the concept of  relative entropy in \cite{LMMS}.
\end{example}

Now we propose an interpretation of the information gain by using transfer operators defined from {\it a priori} probability kernels.

\begin{remark} Let $\mu$ be an equilibrium probability  and  suppose that $e^{B}$ is the Lipschitz $\nu-$Jacobian of $\mu$. Consider the identification of $\Omega$ and $X\times Y$ given by \eqref{omegatoxy} and then define an {\it a priori} probability kernel on $\Omega$ by $\hat{\nu}^y(da) = e^{B(a,y)}d\nu(a)$,\,\, $y=|\cdot,x_2,x_3,x_4,...)$.
For a fixed $\hat{\nu}-$normalized function $A$, let $H_A$ be the operator acting on bounded and measurable functions $f:\Omega\to\mathbb{R}$ by
\[H_A(f)(y) = \int e^{A(a,y)}f(a,y)\hat{\nu}^{y}(da),\,\,y=|x_1,x_2,x_3,...). \]
Let $\eta$ be a shift-invariant probability on $\Omega$ and suppose there exists a $\hat{\nu}$-normalized function $A$ such that $H_A^*(\eta) = \eta$. This means that for any measurable function $f$, we have
\[\int e^{A(a,y)}f(a,y)\,\hat{\nu}^y(da)d\eta(y) = \int f(y)\,d\eta(y).\]
Then, $h(\eta,\mu)=IG(\eta,\mu) =IG(\eta,\hat{\nu})= -H^{\hat{\nu}}(\eta) = \int A\,d\eta$. Furthermore, $$IG(\eta,\mu) = \sup\{\int c\, d\eta\,|\, c\, is\, \hat{\nu}-\text{normalized}\}.$$	
\end{remark}

\bigskip

We finish this section by recalling some results presented in \cite{ACR}. Let $\eta,\mu$ be two probabilities on $\Omega$. For each $\Gamma \subset \mathbb{N}$  consider the canonical projection $\pi_\Gamma:\Omega \to M^{\Gamma}$ and, for each $n\in \mathbb{N}$, denote by $\Lambda_n$ the set $\{1,...,n\}$. Moreover, denote by  $\mathcal{A}_n$, the $\sigma$-algebra on $\Omega$ generated by the projections $\{\pi_\Gamma,\,\Gamma \subset \Lambda_n\}$. Denote also
\[\mathcal{H}_{\Lambda_n}(\eta\,|\,\mu) = \left\{ \begin{array}{ll}
\int_\Omega \frac{d\eta|_{\mathcal{A}_n}}{d\mu|_{\mathcal{A}_n}}\log\left(\frac{d\eta|_{\mathcal{A}_n}}{d\mu|_{\mathcal{A}_n}}\right)\,d\mu,& \,\,\text{if} \,\,\eta \ll \mu \,\,\text{on}\,\,\mathcal{A}_n \\
+\infty &\,\, else \end{array} \right. . \]

The next result is a consequence of Theorems 1 and 3 in \cite{ACR}. From this result, we get an alternative and equivalent way of  extending the concept of  specific information gain for the TFCA model by considering \eqref{eq2} instead \eqref{eq3} and \eqref{IGall2}.

\begin{proposition}
If $\mu$ is an equilibrium probability and $\eta$ is shift-invariant on $\Omega$, then
\[\lim_{n\to\infty} \frac{1}{n} \mathcal{H}_{\Lambda_n}(\eta\,|\,\mu) = h(\eta,\mu).\]
\end{proposition}

\section{The Involution kernel and the entropy production in the TFCA model}\label{secEP}

In the same way as in last section we assume that $M$ is a compact metric space. We denote by $\Omega^{-}$ the space $M^{\mathbb{N}}$ with elements written in the form $y=(...,y_{3},y_{2},y_1|,\,\,y_i \in M$, and using the same metric as the one previously defined in $\Omega=\Omega^+$.

Points in  $\hat{\Omega}=\Omega^-\times \Omega^+$ are written in the form
$$(y\,|\,x)=(...,y_{3},y_{2},y_1|x_1,x_2,x_3,...).$$ 
The bidirectional shift map $\hat{\sigma}:\hat{\Omega}\to\hat{\Omega}$ is defined by (\ref{ret}).
The restrictions  of  $\hat{\sigma}$ to $\Omega^{+}$ and $\Omega^{-}$ are denoted, respectively,  by $\sigma$ and $\sigma^-$.

Observe that  $(\Omega^-,\sigma^-)$  can be identified with $(\Omega,\sigma)$ from the conjugation $\theta: \Omega^-\to \Omega$, given by $\theta((...,z_{3},z_{2},z_1|) = |z_1,z_2,z_3,...)$. Using this conjugation any result previously stated for $(\Omega,\sigma)$ has an analogous claim for $(\Omega^{-},\sigma^-)$.

\bigskip

Consider a Lipschitz function $A:\Omega^-\times \Omega \to \mathbb{R}$, which does not depend of $y\in \Omega^-$. Then, it is naturally expressed as $A(x)=A\,(\, |x_1,x_2,x_3,...)\,)$.  One can show that there exists a (several, in fact) Lipschitz function $W:\Omega^-\times \Omega \to\mathbb{R}$, which is called an \textbf{involution kernel}, and a Lipschitz function $A^{-}$, such that
\begin{equation}\label{Aminus}
A^{-} := A\circ \hat{\sigma}^{-1} + W\circ\hat{\sigma}^{-1} -W,
\end{equation}
where the function $A^{-}$ {\bf does not depend on $x\in \Omega$} (see \cite{BLT},\cite{LMMS}). The action of $A^{-}$   is naturally expressed in coordinates $y= (...,y_3,y_2,y_1| $ as $y \to A^{-}(y)$ and the action of $W$ can be expressed as  $(y\,|\,x) \to W(y\,|\,x).$

All the above can be written in the form:
\[A^{-}(y)=A^{-}((...,y_3,y_2,y_1|) = A(|y_1,x_1,x_2,...) \,)+ W(...,y_3,y_2|y_1,x_1,x_2,x_3,...)\]
\begin{equation}\label{Aminus1}-   W(...,y_3,y_2,y_1|x_1,x_2,x_3,...),\end{equation}
for any $(...,y_3,y_2,y_1|x_1,x_2,x_3,...) \in \hat{\Omega}$.

We point out that  for several important examples of potentials $A$ it is possible to get the involution kernel  $W$ in an explict form (see \cite{CDLS}, \cite{BLT}, \cite{BLL}).

Following \cite{BLT} and \cite{LMMS} we state  two propositions.

\begin{proposition}
	Let $A:\Omega^{+}\to\mathbb{R}$ be a Lipschitz function and $W:\hat{\Omega}\to\mathbb{R}$ be a Lipschitz involution kernel for $A$. Consider the function $A^-$ which was defined by (\ref{Aminus}). Fix an a priori probability $\nu$ on $M$. Then, for any $x\in \Omega^+$, $y\in \Omega^-$ and any function
	$f:\hat{\Omega}\to\mathbb{R}$,
	\begin{equation}
	\mathcal{L}_{A^-,\nu}\left(f(\cdot|x)
	\, e^{W(\cdot|x)}\right)(y)
	=\mathcal{L}_{A,\nu}\left(f\circ\hat\sigma(y|\cdot)
	\, e^{ W(y|\cdot)}\right)(x).
	\end{equation}
	\end{proposition}

\bigskip


\begin{proposition}\label{prop1} 	Let $A:\Omega^{+}\to\mathbb{R}$ be a Lipschitz function and $W:\hat{\Omega}\to\mathbb{R}$ be a Lipschitz involution kernel for $A$. Consider the function $A^-$ as defined by (\ref{Aminus}). Fix an a priori probability $\nu$ on $M$. Let $\rho_A$ and $\rho_{A^-}$ be the eigenmeasures for $\mathcal{L}^*_{A,\nu}$ and $\mathcal{L}^*_{A^-,\nu}$, respectively. Suppose $c$ is such that
	$\iint \, e^{ W(y|x)-c}\, d\rho_{A^-}(y) d\rho_A(x)=1$, and denote  $K(y|x):=e^{ W(y|x)-c}$. Then,\newline

\noindent
1. The probability	
	$$d\,\hat{\mu}_A= K(y|x)\,d\rho_{A^-}(y)\,\, d\rho_{A}(x)$$
is invariant for $\hat{\sigma}$ and it is an extension of the $\nu$-equilibrium probability $\mu_A$.\newline
2. 	The function $h_A(x)=\int K(y|x)\,d\rho_{A^-}(y)$ is  the main eigenfunction for $\mathcal{L}_{A,\nu},$ and
	the function $h_{A^-}(y)=\int K(y|x)\,d\rho_{A}(x)$ is  the main eigenfunction for $\mathcal{L}_{A^-,\nu}.$ \newline
3.  $\lambda_A=\lambda_{A^-}.$
	\end{proposition}

In short it can be said that  the function $(y|x) \to e^{ W(y|x)-c}$ is an integral kernel that connects dual objects: the eigenfunction and the eigenprobability for the Ruelle operator.

Now we apply these results in the understanding of the concept of entropy production. We start by refining  item 1. of the last proposition.

\begin{proposition}\label{prop2}
The probability 	$d\,\hat{\mu}_A= K(y|x)\,d\rho_{A^-}(y)\,\, d\rho_{A}(x)$ is the  unique  $\hat{\sigma}$-invariant extension to $\hat{\Omega}$ of the equilibrium probability $\mu_A$ on $\Omega$.
\end{proposition}

\begin{proof} Let $\hat{\mu}$ be any $\hat{\sigma}$-invariant probability on $\hat{\Omega}$ satisfying $\int g\,d\hat{\mu} = \int g\, d\mu_A$, when  $g(y|x)$ does not depend of $y$. Consider  any continuous function $f$ on $\hat{\Omega}$. We claim  that $\int f\, d\hat{\mu} = \int f\, d\hat{\mu}_A$. Indeed,  as $\hat{\Omega}$ is compact, the function $f$ is uniformly continuous. Fix any point $y_0\in M$ and define the functions $f_n$ on $\hat{\Omega}$, $n \in \mathbb{N}$, by
$f_n(y|x) = f(y^n|x)$, where $y^n = (...,y_0,y_0,y_0,y_n,y_{n-1},...,y_2,y_1|$.

It follows that $\{f_n\}$ converges uniformly to $f$, and moreover, the function
$f_n((...,y_3,y_2,y_1|x_1,x_2,x_3,...))$
does not depend of $y_k$, for $k>n$.

From,
\[ \int f_n\, d\hat{\mu} = \int f_n \circ \hat{\sigma}^{-n}\,d\hat{\mu} =  \int f_n \circ \hat{\sigma}^{-n}\,d{\mu_A} =  \int f_n \circ \hat{\sigma}^{-n}\,d\hat{\mu}_A = \int f_n\, d\hat{\mu}_A,\]
we conclude that  $\int f\, d\hat{\mu} = \int f\, d\hat{\mu}_A$.
\end{proof}

\begin{notation} Let $\mu$ be an equilibrium probability on $\Omega^{+}$. We denote by $\hat{\mu}$ the unique  $\hat{\sigma}$-invariant extension to $\hat{\Omega}$ of $\mu$ and by $\mu^-$ the restriction of $\hat{\mu}$ to $\Omega^{-}$.
\end{notation}

\begin{proposition}\label{muminuseq} Let $A:\Omega^+\to\mathbb{R}$ be a Lipschitz function and  $W$ be any Lipschitz involution kernel for $A$. Now, consider the function $A^{-}$ on $\Omega^{-}$ as defined by (\ref{Aminus}). Fix an a priori probability $\nu$ on $M$. Let $\mu_A$ be the $\nu-$equilibrium of $A$ and let $(\mu_A)^-$ defined as above. Then,  	
 	$(\mu_A)^-$ is the $\nu-$ equilibrium of $A^-$ in $\Omega^-$, that is
 	\[(\mu_A)^- = \mu_{(A^-)}.\]
\end{proposition}
\begin{proof}
From the above  	
$$d\,\hat{\mu_A}= K(y|x)\,d\rho_{A^-}(y) d\rho_{A}(x),$$ and
$h_{A^-}(y)=\int K(y|x)\,d\rho_{A}(x)$ is  the main eigenfunction for $\mathcal{L}_{A^-,\nu}$.
Then, for any continuous function $f:\Omega^{-}\to\mathbb{R}$ we get
\[\int f(y)\,d(\mu_A)^- = \int f(y) \,d\hat{\mu_A} = \iint f(y) K(y|x)\,d\rho_{A}(x)d\rho_{A^-}\,(y) \]
\[= \int f(y) h_{A^-}(y)d\rho_{A^-}\,(y) = \int f(y) \,d\mu_{A^-}.\]

\end{proof}

\begin{definition} The {\bf entropy production} of the  equilibrium probability $\mu$ is defined as
	\[ e_p(\mu) = h(\mu,\theta_*\mu^-),\]
	where $\theta_*\mu^-$ on $\Omega^{+}$ is the push-forward of $\mu^{-}$ by the conjugation $\theta:\Omega^-\to\Omega^+$ given by \eqref{ket}.
\end{definition}	

Observe that as a consequence of the variational principle we get  $e_p(\mu)\geq 0$, and it is zero, if and only if, $\mu^-=\mu$. As the specific information gain $h(\mu,\mu^-)$ ``does not depend of $\nu$'', the above definition also ``does not depend of $\nu$''. In fact, by definition, we should have to consider the {\it a priori} probability $P_{\mu^-}$, but, if for some {\it a priori} probability $\nu$ the measure $\mu$ is a $\nu$-equilibrium probability, then it follows that the probability $\mu^-$ also satisfies this property. Now, applying Proposition \ref{anynu} we get an alternative formula for computing  expression $e_p(\mu)$, but now using the {\it a priori} probability $\nu$.  
%
%

We will exhibit below other alternative ways for computing the entropy production.
Let $\hat{\theta}:\hat{\Omega}\to\hat{\Omega}$ be given by
\begin{equation} \label{bra} \hat{\theta}(...,y_3,y_2,y_1|x_1,x_2,x_3,...)=(...x_3,x_2,x_1|y_1,y_2,y_3,...).
\end{equation}

Observe that $\hat{\theta}^{-1}=\hat{\theta}$ and $\hat{\theta} \circ \hat{\sigma}^{-1} = \hat{\sigma}\circ\hat{\theta}$.

\begin{proposition}\label{equal} Let $A:\Omega\to\mathbb{R}$ be a Lipschitz function, $W:\hat{\Omega}\to\mathbb{R}$ be any Lipschitz involution kernel for $A$ and let $A^-:\Omega^-\to\mathbb{R}$ be defined by (\ref{Aminus}).  Let $\mu$ be any  equilibrium probability on $\Omega$ and consider $\hat{\mu}$ and $\mu^-$ defined as above. Then,
	
	\vspace{0.3cm}
	
	1. $\int A\, d\mu = \int A^-\,d\mu^-$
	
	\vspace{0.3cm}
	
	2. $\int A^-\circ \theta^{-1}\, d\mu = \int A\circ \theta \,d\mu^-$.
\end{proposition}	

\begin{proof}
	In order to prove  item 1. we observe that
	\[
	\int A\, d\mu = \int A \,d\hat{\mu}  = \int A\circ\sigma^{-1} + W\circ\hat{\sigma}^{-1} - W \,d\hat{\mu}
	= \int A^- \,d\hat{\mu}   = \int A^- d\mu^-. \]
	
Now	we will  prove  item 2.
	\[ \int A^-\circ\theta^{-1} \,\,d\mu  = \int A^-\circ\hat{\theta} \,\,d\hat{\mu}   = \int \, A\circ \hat{\sigma}^{-1}\circ \hat{\theta} \,+\, W\circ\hat{\sigma}^{-1}\circ \hat{\theta} \,-\,W\circ \hat{\theta} \,\,d\hat{\mu}\]
	\[  =\int  A \circ \hat{\theta} \circ\hat{\sigma} + W\circ \hat{\theta}\circ\hat{\sigma} -W\circ \hat{\theta} \,\,d\hat{\mu} =\int  A \circ \hat{\theta}\,\,d\hat{\mu} = \int A\circ \theta \,d\mu^-. \]
\end{proof}

\begin{proposition} Let $\mu$ be an equilibrium probability and $\nu$ be an a priori probability. Then, $h^\nu(\mu) = h^{\nu}(\mu^-)$.
\end{proposition}

\begin{proof} For each Lipschitz function $A:\Omega^{+}\to\mathbb{R}$ we can consider a Lipschitz involution kernel $W$, and then, we get an associated  Lipschitz function $A^{-}:\Omega^{-}\to\mathbb{R}$.

For the fixed {\it a priori} probability $\nu$ we have $\lambda_A=\lambda_{A^-}$. Then,
\[h^{\nu}(\mu) = -\sup_{A\, is \, \nu-normalized} \int A\,d\mu = -\sup_{A\, is \,Lipschitz\,on\,\Omega^+} \int A\,d\mu - \log(\lambda_A)\]
\[ = -\sup_{A^-\, given \,from\,some\,Lip.\,A^+ } \int A^-\,d\mu^- - \log(\lambda_{A^-})\]
\[\geq -\sup_{B^-\,is\,Lipschitz \,on \,\Omega^- } \int B^-\,d\mu^- - \log(\lambda_{B^-})= h^{\nu}(\mu^-).\]

In order to get the opposite inequality, we follow a similar argument. We exchange the reasoning by $\hat{\theta}$: for each Lipschitz function $B^-:\Omega^-\to\mathbb{R}$, we take an involution kernel, and, an associated  Lipschitz function $B^+:\Omega^{+}\to\mathbb{R}$.  Now, we just have to proceed in the same way as before.
	
\end{proof} 	

\bigskip



As a consequence we get  the following claim for the entropy production:
\begin{proposition}\label{variational}
Suppose that $\mu$ is an equilibrium probability and consider the associated probability $\mu^-$. Suppose that for an a priori probability $\nu$ and for a Lipschitz function $A^-$ we have that $\mu^-$ is the $\nu$-equilibrium probability for  $A^-$. Now, assume  that $\mu^-$ and $A^-$ are defined on $\Omega^+$ via the conjugation $\theta$. Then, the entropy production of $\mu$ satisfies
	\[e_p(\mu) =   \int A^-\,d\mu^- - \int A^-\,d\mu.\]
We can take $A^-$, such that, $J^-=e^{A^-}$ is the $\nu-$Jacobian of $\mu^-$.	
\end{proposition}

\begin{proposition}\label{variational2} Suppose that $\mu$ is the $\nu$-equilibrium probability for the Lipschitz function $A:\Omega^+\to\mathbb{R}$. Let $W$ be any Lipschitz involution kernel for $A$ and $A^-:\Omega^-\to \mathbb{R}$ be the function defined by (\ref{Aminus}). Suppose that $A^{-}$ is defined on $\Omega^+$ using the conjugation $\theta$. Then,
	\[e_p(\mu) =  \int A - A^-\,d\mu .\]
\end{proposition}

\begin{proof} The claim follows from the previous result and Proposition \ref{equal}.
\end{proof} 	

\begin{definition}\label{sim} Given the potential $A$, suppose that $A^{-}$ is defined on $\Omega^+$ using the conjugation $\theta$. We say that the potential $A$ is symmetric if
$A = A^- $.
\end{definition}

\begin{corollary}\label{variational21} Suppose for that some  involution kernel 
$W$ the potential  $A:\Omega \to \mathbb{R}$ is symmetric, then, the equilibrium probability for $A$ has entropy production zero.
\end{corollary}

There are several examples of potentials $A$  that are symmetric (see for instance \cite{CDLS}, \cite{BLT}, \cite{BLL}). Note that in order to check if the equilibrium probability $\mu$ for the Holder potential $A$ has entropy production zero one have to follow a process of finding the eigenfunction and the eigenprobability; which is a  procedure that in general we do not have explicit expressions.
All this can be avoided when it is possible to show that for some involution kernel the potential is symmetric.

\begin{proposition} Suppose that $\mu$ is an equilibrium probability. Then,
	\[e_p(\mu) = h(\mu,\mu^-)  = h(\mu^-,\mu) = e_p(\mu^-).\]
\end{proposition}
\begin{proof} It follows from Proposition \ref{prop1} and \ref{prop2} that $(\mu^{-})^- = \mu$. Consider an {\it a priori} probability $\nu$, such that, $\mu$ is the $\nu-$equilibrium probability for a Lipschitz function $A$. Let $A^-$ defined by (\ref{Aminus}) using any involution kernel. From Proposition \ref{variational} we get
	\[e_p(\mu) = h(\mu,\mu^-)= \int A^-\,d\mu^- - \int A^-\circ\theta^{-1}\,d\mu\]
	and
	\[e_p(\mu^-) = h(\mu^-,\mu) = \int A\,d\mu - \int A\circ\theta\,d\mu^- . \]
	Now, from  proposition \ref{equal} we get
	\[\int A^-\,d\mu^- - \int A^-\circ\theta^{-1} \,d\mu  = \int A\,d\mu - \int A\circ\theta\,d\mu^-.\]
	This ends the proof.
\end{proof}		


The next example considers the more simple case where $M=\{1,2,...,d\}$ is a finite set.

\begin{example} Take $M=\{1,2,...,d\}$ and consider as the {\it a priori} measure  $\nu$  the counting measure on $M$.

Any invariant probability $\mu$ for $(\Omega,\sigma)$ 
%
%
 can be extended to a $\hat{\sigma}$-invariant probability $\hat{\mu}$ on $\hat{\Omega}$ by defining
\[\hat{\mu} ([a_m,...,a_1|b_1,...,b_n]) := \mu(\,|a_m,...,a_1,b_1,...,b_n]\,), \]
and using the extension theorem.
The restriction of $\hat{\mu}$ to $\Omega^{-}$  satisfies
\[\mu^{-}(\,[a_m,...,a_{2},a_1|\,) = \mu(\,|a_m,...,a_{2},a_1]\,).\]

Now, using the conjugation $\theta: \Omega^-\to \Omega$ in order to transfer $\mu^-$ to $\Omega^+$, we get
\begin{equation}\label{eq1}
\theta_*\mu^-(|a_1,a_2....a_m]) = \mu(|a_m,...,a_{2},a_1]).
\end{equation}
As the Kolmogorov-Sinai entropy of $\mu$ is given by
\[h(\mu)=\lim_{n\to\infty}-\frac{1}{n}\sum_{i_1,...,i_n}\mu(|i_1,...,i_n])\log(\mu(|i_1,...,i_n]) ),\]
we conclude that $h(\mu)=h(\mu^{-})$.


\bigskip
 Suppose now, that $\mu$ is the equilibrium probability for the Lipschitz normalized potential $A$. Then, $e^A=J$ is the Jacobian of $\mu$, that is,
\[ e^A(|x_1,x_2,x_3,...))=J(|x_1,x_2,x_3,...)) = \lim_{n\to\infty} \frac{\mu(|x_1,x_2,x_3,...,x_n])}{\mu(|x_2,x_3,...,x_n])}. \]
Let $J^-$ be the Jacobian of $\mu^-$ and define $A^- := \log(J^-)$. 
Then,  using (\ref{eq1}),
\[ e^{A^-}(...,y_3,y_2,y_1|) = J^-(...,y_3,y_2,y_1|) = \lim_{n\to\infty} \frac{\mu(|y_n,...,y_2,y_1])}{\mu(|y_{n},...,y_2])}. \]

\end{example}

\bigskip

The next example computes  the entropy production for a Markov measure $\mu$. Our estimate is coherent with  expression   (1) in \cite{Jiang}.

\begin{example}
	Consider the line stochastic matrix $M=(p_{ij})$ and the initial probability vector $P=(\pi_i)$, such that, $PM=P$.
	
	We denote by $\mu$ the associated Markov measure, that is, for any cylinder $|x_1,x_2,...,x_n]$ we set
	\[\mu(|x_1,x_2,...,x_n])=\pi_{x_1}\cdot p_{x_1x_2}\cdots p_{x_{n-1}x_n}.\]
	Then,
	\[ J(|i,j,x_3,...) = \frac{\pi_{i}p_{ij}}{\pi_j}. \]
	We also  get
	
	$$J^-(...,y_3,j,i|) = \lim_{n\to\infty} \frac{\mu(|y_n,...,y_3,j,i])}{\mu(|y_{n},...,y_3,j])} =$$
	 $$\lim_{n\to\infty} \frac{\pi_{y_n}\cdot p_{y_ny_{n-1}}\cdots p_{y_3j}\cdot p_{ji}}{\pi_{y_n}\cdot p_{y_ny_{n-1}}\cdots p_{y_3j}}=p_{ji}.   $$
	
	As $J^-$ depends only on two coordinates, $\mu^-$ is also a Markov measure.
	
Considering the conjugation $\theta$, we get,
	\[\mu(|i,j])=\pi_{i}p_{ij} \,\,\,and \,\,\,\mu^-(|i,j]) = \mu(|j,i])=\pi_{j}p_{ji}. \]
	Taking $A=\log(J)$ and $A^-=\log(J^-)$, we also get
	\[ e^{A(|i,j,x_3,...)} = \frac{\pi_{i}p_{ij}}{\pi_j}\,\,\,
	and
	\,\,\,e^{A^-(|i,j,z_3,z_4,...))} =  p_{ji}. \]
	Then, using the Proposition \ref{variational2}, we derive
	\[e_p(\mu) = \int A - A^-\,d\mu = \sum_{i,j} \log\left(\frac{\pi_{i}p_{ij}}{\pi_j p_{ji}}\right)\pi_{i}p_{ij}. \]
	
	\bigskip
	
	We can compute $e_p(\mu)$, alternatively, using Proposition \ref{variational}:
	\[e_p(\mu) = \int A^-\,d\mu^- - \int A^-\,d\mu =  \sum_{i,j}\log(p_{ji})\mu^-(|i,j))-\sum_{i,j}\log(p_{ji})\mu(|i,j))\]
	\[=  \sum_{i,j}\log(p_{ji})\pi_jp_{ji}-\sum_{i,j}\log(p_{ji})\pi_i p_{ij}
	=  \sum_{i,j}\log(p_{ij})\pi_i p_{ij}-\sum_{i,j}\log(p_{ji})\pi_i p_{ij} \]
	\[= \sum_{i,j}\log(p_{ij})\pi_i p_{ij}-\sum_{i,j}\log(p_{ji})\pi_i p_{ij} + \left[\sum_{i}\pi_i\log(\pi_i) - \sum_{j}\pi_j\log(\pi_j)\right] \]
	\[= \sum_{i,j}\log(p_{ij})\pi_i p_{ij}-\sum_{i,j}\log(p_{ji})\pi_i p_{ij} + \left[ \sum_{i,j}\pi_i p_{ij}\log(\pi_i) - \sum_{i,j}\pi_i p_{ij}\log(\pi_j) \right]\]
	\[=\sum_{i,j}\log(\frac{\pi_i p_{ij}}{\pi_jp_{ji}})\pi_ip_{ij}.\]

	 In this case  an involution kernel for $A$ is the function $W: \{1,2\}^\mathbb{N} \to \mathbb{R}$ given by
$$W ( ... ,y_2, y_1 | x_1,x_2,..) =
\log p_{y_1\,x_1} - \log \pi_{ x_1}$$
and the corresponding $A^{-}$ is given by the $A^{-} (i ,j ,y_3,y_4..)=p_{j \,i}.$
	
\end{example}

The case with just two symbols is quite special as we will see now.

\medskip

\begin{example}  Entropy production  zero -     Suppose $\Omega=\{1,2\}^{\mathbb{N}}$ and  assume that $\mu$ is a Markov measure (as defined above). Then, $e_p(\mu) = 0$.
	Indeed, as $\mu$ is invariant we get $\mu(|1,2)) = \mu(|2,1))$, and therefore, $\mu^-(|i,j))=\mu(|j,i)) = \mu(|i,j)$, for any $i,j\in\{1,2\}$. It follows that $J^- = J^+$, and therefore, $\mu^-=\mu$. Consequently,
\[e_p(\mu) = \int \log(J) - \log(J^-)\,d\mu = 0. \]

That is, in this case,  the entropy production is zero.
\end{example}

It follows from Corollary 2.3 in \cite{W} that this result -  entropy production  zero - also happen for equilibrium probabilities of a more general class of functions defined on $\Omega=\{1,2\}^{\mathbb{N}}$ (see \cite{W}).

The probability described in \cite{LMMM}  also has entropy production zero (see section 2 in \cite{LMMM}).

Markov measures on
$\Omega=\{1,2,3\}^{\mathbb{N}}$ may have non zero entropy production.

\section{Appendix: Examples in information theory}\label{appendix1}

Our intention in this section is to illustrate the theoretical results which we previously described concerning  Shannon entropy $S(P)$ and information gain $IG(\pi,P)$ via worked examples (a nice general reference on the topic  is \cite{CT}). We believe that this short presentation  will be helpful for mathematicians that do not have much familiarity with these concepts.

We start by considering the Shannon entropy  which is sometimes alternatively called   mean information.

The number $S(P)$ can be interpreted (taking basis 2 for the logarithm) as a lower bound for the average of questions of type ``yes or no'' which are necessary in order to analyze the statistics of  a symbol picked at random - according to the probability distribution $P=(p_1,...,p_d)$ - on the finite alphabet $\{1,...,d\}$. From the sequence of answers to  successive questions - of a certain type - one can introduce a binary code on the set $\{1,...,d\}$, where $0$ corresponds to ``yes'' and $1$ to ``no'' (see \cite{CT} chap. 5).

\begin{example}  Suppose that a box has balls of 4 possible different colors. Two people will play a game with the following rules: one ball is picked off the box by one of them and the other person must discover the color of this ball by making questions of the type ``yes or not''.

If this game is repeated several times, the balls are picked randomly according with the probability $P=(p_1,p_2,p_3,p_4) = (\frac{1}{4},\frac{1}{4},\frac{1}{4},\frac{1}{4})$ and the strategy used for the questions is optimal, what is the mean value of the number of questions which are necessary?
	
	\bigskip
	
We will replace the colors with symbols of the set $\{ 1,2,3, 4\}$. One can consider the following strategy of questions:\newline
	Q1: is the picked symbol 1 or 2? \newline
	- with the answer ``yes'' it can be considered the question Q2: is the symbol 1? \newline
	- with the answer ``no'' it can be considered the question Q2': is the symbol 3? \newline
Using this strategy it is necessary exactly two questions in order to discover the  symbol (color) which was taken. It  coincides with the Shannon entropy (the mean information)
	\[ S(P)=-\sum_{i=1}^{4}\frac{1}{4}\log_2(\frac{1}{4})=2. \]
Observe that the set of symbols $\{1,2,3,4\}$ can be encoded as the answers $(yy,yn,ny,nn)$. Replacing $y$ by $0$ and $n$ by $1$ we can encode $\{1,2,3,4\}$ as $(00,01,10,11)$ in binary expansion, which is optimal.
\end{example}

\begin{example} Proceeding as in above example, but now assuming  that the colors of the balls are picked randomly according to the probability $P=(p_1,p_2,p_3,p_4) = (\frac{1}{2},\frac{1}{4},\frac{1}{8},\frac{1}{8})$, one can use the following strategy of questions: \newline
	Q1: is the symbol (color) 1? (with probability (frequency) $\frac{1}{2}$ this unique question solves the problem)  \newline
	- with the answer ``yes'' we finish.\newline
	- with the answer ``no'' we consider the question Q2: is the symbol 2? \newline
	- with the answer ``yes'' we finish.\newline
	- with the answer ``no'' again, we then consider the question Q3: is the symbol 3? \newline
If this game is repeated several times, using this strategy the mean number of questions is:
	\[(1\,\text{question})\frac{1}{2} + (2\, \text{questions})\frac{1}{4} + (3\, \text{questions})\frac{1}{4} = \frac{7}{4}.\]
It coincides  with the Shannon entropy (mean information)
	\[ S(P)=-[\frac{1}{2}\log_2(\frac{1}{2}) + \frac{1}{4}\log_2(\frac{1}{4}) + \frac{1}{8}\log_2(\frac{1}{8}) +  \frac{1}{8}\log_2(\frac{1}{8})] =\frac{7}{4}. \]
	In this case $\{1,2,3,4\}$ can be encoded as $\{0,10,110,111\}$ in binary expansion, being this one optimal.
\end{example}

\begin{example} Proceeding as above and supposing that there are only two colors of balls which are picked randomly according with the probability $P=(p_1,p_2) = (\frac{2}{3},\frac{1}{3})$ one can consider the following question: \newline
	Q1: is the color (symbol) 1? \newline
With this strategy, the mean number of questions is exactly 1 which is bigger than the Shannon entropy $S(P) \approx  0,918$.  In this case $\{1,2\}$ can be encoded as $\{0,1\}$ in binary expansion.
\end{example}

We refer to \cite{CT} chap. 5 for a more complete discussion of the topic. Our intention above was  just to illustrate - with introductory and simple examples - the fact  that  the Shannon entropy is as a lower bound for the average number of questions and how one can introduce a binary code for a set of symbols $\{1,...,d\}$.

\bigskip

From now we will discuss an example concerning the Information Gain $IG(\pi,P)$ (or mutual information). We refer to \cite{Q} (see p. 89-90) for a more detailed discussion of this topic in the context of decision trees in Machine Learning.

\begin{example}\label{exampleIG}
Consider - in a similar way as before - a box with a collection of 100 objects, being 30 of them of the color blue and 70 of them of the color red. It's also known that:  \newline
a. 10 of the blue objects are balls and 20 of them are cubes\newline
b. 45 of the red objects are balls and 25 of them are cubes.

Considering  all this set of  information we can construct probabilities $P$ and $\pi$ in the following way:
\[\begin{array}{c} 30\,\, blue \\ 70\,\, red \end{array} \to P =\begin{pmatrix} 0.3\\0.7\end{pmatrix}  \hspace{1cm}
\begin{array}{ccc}  & balls & cubes \\
blue&  10 & 20    \\
red&  45 & 25
\end{array} \to \pi=\begin{pmatrix} 0.10&0.20\\0.45&0.25\end{pmatrix}.\]

We consider that $\pi$ is defined in a Cartesian product $X\times Y$ and has $x-$marginal $P=(\frac{30}{100}, \frac{70}{100})$ (adding in the lines of $\pi$) and $y-$marginal $Q= (\frac{55}{100}, \frac{45}{100})$ (adding in the rows of $\pi$).

\bigskip

We will consider two kinds of different games.

\bigskip

\textbf{Game one:} One object is randomly picked of the box and we shall discover its color by asking questions of the  type yes or no. In this case the Shannon's entropy, or mean information, is    equal to
\[S(P)=-\left[\frac{30}{100}\log(\frac{30}{100}) + \frac{70}{100}\log(\frac{70}{100})\right].\]

\bigskip

\textbf{Game two:} In this game - in a similar way as in game one - we have the same goal. However, in the present game, after the object was picked we receive  partial  information about the result, which is: ``it is a cube'' or ``it is a ball''.

In this game, with probability (or, frequency)  $\frac{55}{100}$, the information to be received it will be that it was picked a ball. Using this information we must concentrate our attention for such class of objects and so the colors are distributed according to the probability $(\frac{10}{55}, \frac{45}{55})$. Similarly, with probability (frequency) $\frac{45}{100}$, the information received will be that a cube was picked. In this case, we consider the colors distributed according  to the probability  $(\frac{20}{45}, \frac{25}{55})$. Therefore, the mean information in this game is given by a weighted mean of two Shannon's entropies, that is,
\[H(\pi) =-\frac{55}{100}\left[\frac{10}{55}\log(\frac{10}{55}) + \frac{45}{55}\log(\frac{45}{55})\right] - \frac{45}{100}\left[\frac{20}{45}\log(\frac{20}{45}) + \frac{25}{45}\log(\frac{25}{45})\right]. \]

\bigskip

Finally, we observe that the information gain $IG(\pi,P)$ given in \eqref{IG} is the difference between the mean information in game one and the mean information in game two,
	\begin{align*} IG(\pi,P)& = S(P)-H(\pi) .
	\end{align*}
\end{example}

\medskip

\section{Appendix: Variational form of $H(\pi)$}\label{appendix2}
	
In this section, we propose to study the entropy $H(\pi)$ which appears in \eqref{IG} in a similar way as in \cite{M}.

If $(a_1,...,a_n)$ and $(b_1,...,b_n)$ are probability vectors such that $b_i >0,\,\forall\, i$, then,
\begin{equation}\label{basicineq}
\sum_{i=1}^{n} a_i\log(a_i) \geq \sum_{i=1}^{n}a_i\log(b_i) ,
\end{equation}
with equality only if $a_i=b_i,\, \forall i$. This classical result  can be found for example in \cite{PP} lemma 3.3.

We will say that $f:X\times Y \to\mathbb{R}$ is a \textbf{normalized function}, if it satisfies
\[ \sum_{x\in X} e^{f(x,y)} = 1, \, \forall y.\]
If the probability $\pi$ on $X\times Y$ satisfies $\pi_{x,y}>0,\,\forall (x,y)$, then $\log(J^{\pi})$ is a normalized function.

\begin{proposition}  Let $\pi$ be a probability on $X\times Y= \{1,...,d\}\times\{1,...,r\}$ and $f$ be a normalized function. Then,
	\[ \sum_{x=1}^{d}\sum_{y=1}^{r} {\pi_{x,y}}\log(J^{\pi}(x,y)) \geq  \sum_{x=1}^{d}\sum_{y=1}^{r} {\pi_{x,y}}f(x,y).\]
	The equality occurs only if $J^\pi_{x,y} =e^{f(x,y)}$, $\forall (x,y)$, such that, $\pi_{x,y}>0$.
\end{proposition}

\begin{proof} Let $q_y = \sum_x \pi_{x,y}$.  From \eqref{basicineq}, if $q_y>0$, we have
	\[ \sum_{x} J^{\pi}(x,y) \log (J^{\pi}(x,y)) \geq \sum_{x} J^{\pi}(x,y),\log(e^{f(x,y)}),\]
	with equality only if  $J^\pi_{x,y} =e^{f(x,y)}$, $\forall x$. By definition $J^{\pi}(x,y) = \frac{\pi_{x,y}}{q_y}$, if $q_y>0$,  then we get
	\[ \sum_{x} {\pi}(x,y) \log (J^{\pi}(x,y)) \geq \sum_{x} {\pi}(x,y){f(x,y)}).\]
	If we assume that $J^\pi(x_0,y_0) \neq e^{f(x_0,y_0)}$, for some $(x_0,y_0)$, such that $\pi(x_0,y_0)>0$, then, we get
	\[ \sum_{x,y} {\pi}(x,y) \log (J^{\pi}(x,y)) > \sum_{x,y} {\pi}(x,y){f(x,y)}.\]
\end{proof}

\begin{proposition}\label{Hfinito} Let $\pi$ be a probability on $X\times Y= \{1,...,d\}\times\{1,...,r\}$. Then,
	\[	H(\pi) = - \sup \{ \sum_{x,y} f(x,y)\pi_{x,y}\,|\, \sum_{x\in X} e^{f(x,y)} = 1, \, \forall y\}.\]
\end{proposition}	

\begin{proof}
	If $\pi_{x,y} >0,\, \forall (x,y)$, then $J^{\pi}$ is well defined, normalized and positive in $X\times Y$. From the last proposition, we get that the function $\log(J^{\pi})$ attains the supremum. In this case, the proof is finished.
	If 	$\pi(x_0,y_0)=0$, for some point $(x_0,y_0)$, then the function $\log(J^{\pi})$ is only  well defined for $\pi$ a.e. $(x,y)$. In this case we get, from the last proposition,
	\[	H(\pi) \leq - \sup \{ \sum_{x,y} f(x,y)\pi_{x,y}\,|\, \sum_{x\in X} e^{f(x,y)} = 1, \, \forall y\}.\]
	In order to prove the opposite inequality we consider for each $\epsilon>0$ the function $f^{\epsilon}$ defined in the following way: for fixed $y_0$, if $\pi(x,y_0)>0$, for any $x$, then $f^{\epsilon}(x,y_0) = \log(J^{\pi}(x,y_0)),\, \forall x$.
	For fixed $y_0$, if $\pi(x_0,y_0) = 0$, for some $x_0$,  we define
	\[ f^{\epsilon}(x,y_0) =
	\left\{\begin{array}{cc}
	\log((1-\epsilon)J^{\pi}(x,y_0)) & \text{if} \, \pi(x,y_0)>0 \\
	a(\epsilon,y_0) & \text{if}\, \pi (x,y_0) = 0\end{array}\right.,\]
	where $a(\epsilon,y_0)$ is chosen in such way that $\sum_{x}e^{f(x,y_0)} = 1.$
	
	With this construction we get that $f(x,y)$ is well defined for any $(x,y)\in X\times Y$ and $\sum_{x}e^{f(x,y)} = 1, \,\, \forall y$. Furthermore,
	\begin{align*}
	\sum_{x,y} f^{\epsilon}(x,y)\pi_{x,y} &\geq \sum_{x,y} \log((1-\epsilon)J^{\pi}(x,y_0)) \pi_{x,y}\\
	& =  \log(1-\epsilon) + \sum_{x,y} \log(J^{\pi}(x,y_0)) \pi_{x,y}.
	\end{align*}
	Then,
	\begin{align*}
	H(\pi) &= -  \sum_{x=1}^{d}\sum_{y=1}^{r} \log(J^{\pi}(x,y)){\pi_{x,y}} \\
	&\geq - \sup \{ \sum_{x,y} f(x,y)\pi_{x,y}\,|\, \sum_{x\in X} e^{f(x,y)} = 1, \, \forall y\} -\log(1-\epsilon).
	\end{align*}
	Taking $\epsilon \to 0$, we finish the proof.
\end{proof}

\medskip

A. O. Lopes partially is supported by CNPq grant.

\end{document}